\documentclass[12pt,twoside]{article}
\usepackage[hmargin=0.8in,vmargin=0.9in]{geometry}
\usepackage{fancyhdr}
\usepackage{graphicx}
\usepackage{amssymb}
\usepackage{amsmath}
\usepackage{amsthm}
\usepackage{amsfonts}
\usepackage{mathrsfs}
\usepackage{mathtools}
\usepackage{bm}
\usepackage{color}
\usepackage{setspace}
\usepackage{exscale}
\usepackage{relsize}
\usepackage{float}
\usepackage{cite}
\usepackage{hyperref}
\usepackage{cleveref}
\usepackage{nicefrac}

\usepackage{booktabs,multirow} 
\usepackage{array} 
\usepackage{paralist} 

\theoremstyle{plain}                    
\newtheorem{thm}{Theorem}[section]

\newtheorem{rmk}[thm]{Remark}

\usepackage{tikz}
\usetikzlibrary{positioning}
\usepackage{xcolor}

\allowdisplaybreaks[1]

\numberwithin{equation}{section}
\numberwithin{figure}{section}
\numberwithin{table}{section}

\newcommand\eref[1]{(\ref{#1})}

\newcommand*\xbar[1]{%
  \hbox{%
    \vbox{%
      \hrule height 0.5pt 
      \kern0.4ex
      \hbox{%
        \kern-0.05em
        \ensuremath{#1}%
        \kern-0.00em
      }%
    }%
  }%
}

\usepackage[utf8]{inputenc}
\usepackage[english]{babel}
\usepackage{empheq}
\usepackage{subcaption} 
\usepackage{epsfig}
\usepackage{algorithm}
\usepackage{algpseudocode}
\usepackage{bbm}
\usepackage{sectsty}

\newcommand{\mc}[1]{{\mathcal{#1}}} 

\newcommand{\p}[1]{{\left( #1 \right)}}

\newcommand{\br}[1]{{\left[ #1 \right]}}
\newcommand{\abs}[1]{{\left| #1 \right|}}

\def\hf {\frac{1}{2}}

\newcommand{\kph}{{k+\frac{1}{2}}}
\newcommand{\kmh}{{k-\frac{1}{2}}}
\newcommand{\jph}{{j+\frac{1}{2}}}
\newcommand{\jmh}{{j-\frac{1}{2}}}
\newcommand{\dx}{\Delta x}
\newcommand{\dy}{\Delta y}

\graphicspath{{Figures/}}

\title{A New Locally Divergence-Free Path-Conservative Central-Upwind Scheme for Ideal and Shallow Water Magnetohydrodynamics}
\author{Alina Chertock\thanks{Department of Mathematics and Center for Research in Scientific Computing, North Carolina State University,
Raleigh, NC 27695, USA; {\tt chertock@math.ncsu.edu}},~Alexander Kurganov\thanks{Department of Mathematics, SUSTech International Center for
Mathematics, and Guangdong Provincial Key Laboratory of Computational Science and Material Design, Southern University of Science and
Technology, Shenzhen 518055, China; {\tt alexander@sustech.edu.cn}},~Michael Redle\thanks{Department of Mathematics, North Carolina State
University, Raleigh, NC 27695, USA; {\tt mtredle@ncsu.edu}},~Kailiang Wu\thanks{Department of Mathematics and SUSTech International Center
for Mathematics, Southern University of Science and Technology, and National Center for Applied Mathematics Shenzhen (NCAMS), Shenzhen
518055, China; {\tt wukl@sustech.edu.cn}}}

\begin{document}
\date{}
\maketitle

\begin{abstract}
We develop a new second-order unstaggered path-conservative central-upwind (PCCU) scheme for ideal and shallow water magnetohydrodynamics
(MHD) equations. The new scheme possesses several important properties: it locally preserves the divergence-free constraint, it does not
rely on any (approximate) Riemann problem solver, and it robustly produces high-resolution and non-oscillatory results. The derivation of
the scheme is based on the Godunov-Powell nonconservative modifications of the studied MHD systems. The local divergence-free property is
enforced by augmenting the modified systems with the evolution equations for the corresponding derivatives of the magnetic field components.
These derivatives are then used to design a special piecewise linear reconstruction of the magnetic field, which guarantees a
non-oscillatory nature of the resulting scheme. In addition, the proposed PCCU discretization accounts for the jump of the nonconservative
product terms across cell interfaces, thereby ensuring stability. We test the proposed PCCU scheme on several benchmarks for both ideal and
shallow water MHD systems. The obtained numerical results illustrate the performance of the new scheme, its robustness, and its ability not
only to achieve high resolution, but also preserve the positivity of computed quantities such as density, pressure, and water depth.
\end{abstract}



\section{Introduction}\label{sec1}
This paper focuses on developing a novel numerical method for magnetohydrodynamic (MHD) systems, widely used in many applications, such as
astrophysics, plasma physics, space physics, and engineering. In these models, fluid dynamics equations are coupled with the equations for
the magnetic field, which satisfies the divergence-free condition -- a physically-exact constraint, that is, if initially, the divergence of
the magnetic field is zero, then it must remain zero for all times. When deriving numerical methods for MHD systems, the divergence-free
condition must be handled with care, as neglecting an identically-zero divergence on a discrete level may lead to severe numerical
instabilities and/or nonphysical structures in the numerical solution; see, e.g.,
\cite{balsara1999staggered,Brackbill1980Effect,li2005locally,toth2000constraint}. In addition, like other hyperbolic systems of conservation
and balance laws, the MHD systems typically develop very complicated nonsmooth solution structures containing shock waves, rarefactions, and
contact discontinuities, as well as their interactions.

In the past decades, various numerical techniques have been developed to deal with the divergence-free constraint for MHD systems. An early
effort in this direction is the projection method \cite{Brackbill1980Effect}, which is a post-processing divergence correction procedure
that uses Hodge decomposition to project the non-divergence-free magnetic field into a divergence-free subspace by solving an elliptic
Poisson equation. Another widely used approach is the constrained transport (CT) method, which was proposed in \cite{evans1988simulation}
for simulating MHD flows. This method preserves a specific discrete version of divergence-free condition on staggered grids, and its
variants were further developed by researchers within various frameworks; see, e.g.,
\cite{devore1991flux,dai1998simple,ryu1998divergence,balsara1999staggered,londrillo2004divergence,gardiner2005unsplit,xu2011divergence}.
Unstaggered CT methods were also developed (see, e.g.,
\cite{rossmanith2006unstaggered,helzel2011unstaggered,mishra_tadmor_2012,helzel2013high,christlieb2014finite}), and they are usually based
on numerically evolving the magnetic potential and computing the divergence-free magnetic field through the (discrete) curl of the magnetic
potential. In addition, locally divergence-free discontinuous Galerkin methods that enforce the zero divergence of the magnetic field within
each cell were developed in \cite{li2005locally,Yakovlev2013locally}. In recent years, globally divergence-free high-order methods were
also proposed to enforce the exact zero divergence of the magnetic field within the finite-volume or (central) discontinuous Galerkin
framework; see, e.g.,
\cite{balsara2009divergence,balsara2009efficient,dumbser2019divergence,li2011central,li2012arbitrary,fu2018globally,balsara2021globally}. 

There is also a different class of schemes that reduce the divergence errors but do not explicitly enforce any divergence-free constraint.
In the context of the ideal MHD equations, these methods, typically referred to as eight-wave methods, were proposed by Powell
\cite{powell1995upwind,powell1997approximate,powell1999solution} based on a proper discretization of a modified, nonconservative ideal MHD
model. This model was first introduced by Godunov \cite{Godunov1972symmetric} for entropy symmetrization. Compared
to the conservative ideal MHD equations, the modified model contains extra nonconservative source terms (referred to as Godunov-Powell
source terms in the following), which are proportional to the divergence of the magnetic field. These source terms change the character of
the MHD equations, making the modified model Galilean invariant \cite{dellar2001note}, symmetrizable \cite{Godunov1972symmetric}, and helpful in designing entropy stable schemes (see, e.g., \cite{chandrashekar2016entropy,derigs2018entropy,liu2018entropy}). In
\cite{powell1997approximate}, it was also noticed that the conservative ideal MHD equations are weakly hyperbolic, and thus such source terms
should be added to recover the missing eigenvector. As demonstrated in \cite{powell1995upwind,powell1999solution}, the inclusion of the
source terms ensures that the magnetic divergence is advected with the flow, and the numerical divergence errors are also expected to be
advected and would not accumulate. This makes the eight-wave methods capable of controlling the divergence error, although certain drawbacks
may arise due to a nonconservative nature of the Godunov-Powell modified ideal MHD equations; see \cite{toth2000constraint}. As 
recently discovered in \cite{WuSINUM2018}, a discrete divergence-free condition is closely related to the positivity-preserving property of
numerical schemes for the ideal MHD equations. Furthermore, locally divergence-free positivity-preserving schemes
\cite{WuShu2018,wu2019provably} for the Godunov-Powell modified ideal MHD model can be obtained via geometric quasilinearization
\cite{WuShu2021GQL}. Another class of divergence-controlling schemes is the so-called hyperbolic divergence-cleaning method
\cite{dedner2002hyperbolic}, which introduces a mixed hyperbolic-parabolic equation to damp the divergence errors away instead of enforcing
an exactly divergence-free magnetic field. 

A variant of the MHD equations---known as the shallow water MHD system---has also become a model of significant numerical interest over the
last few decades; see, e.g., \cite{lahaye2022coherent,zeitlin2015geostrophic,petrosyan2020shallow} and references therein. First introduced
in the context of a solar tachocline in \cite{Gilman2000Magnetohydrodynamic} and now used in several astrophysical and geophysical contexts,
this variant is fully derived from the ideal MHD equations under the assumptions of constant density and magneto-hydrostatic equilibrium;
see \cite{Dellar2003Dispersive,Zei13}. An assortment of numerical methods have additionally been explored to treat divergence errors for
this system---such as space-time conservation element solution element methods (CE/SE) in \cite{Ahmed2019higher,Qamar2006Application}, an
evolution Galerkin scheme in \cite{Kroger2005evolution}, and entropy-stable schemes in \cite{Winters2016Entropy,Duan2021High}, to name a
few.

Despite these advances, devising highly accurate, stable, and at the same time, robust numerical methods capable of preserving the
divergence-free condition at a discrete level is still a challenging task. Our main goal is to develop such a scheme. To this end, we
consider the Godunov-Powell modified ideal and shallow water MHD models and supplemented them with additional equations obtained by
differentiating the magnetic field equations in space: the latter will help to ensure local divergence-free conditions. The resulting
augmented MHD systems will be nonconservative and rather complicated to be solved by an upwind numerical method, that is, by a method
relying on a solution of (generalized) Riemann problems. Instead, we numerically solve the augmented MHD systems using second-order
unstaggered finite-volume Riemann-problem-solver-free path-conservative central-upwind (PCCU) schemes, which were introduced in \cite{CKM}
as a black-box solver for nonconservative hyperbolic systems of PDEs. PCCU schemes are a path-conservative extension of the central-upwind
(CU) schemes, which were developed in \cite{kurganov2000new,kurganov2001semidiscrete,kurganov2007reduction} for general multidimensional
hyperbolic systems of conservation laws. We enforce the local divergence-free condition with the help of a special piecewise linear
reconstruction of the magnetic field variables. The resulting scheme produces highly accurate and non-oscillatory results for ideal and
shallow water MHD systems. 

The paper consists of two parts: In \S\ref{sec2}, we study the ideal MHD equations, while \S\ref{sec3} is devoted to the shallow water MHD
system. The Godunov-Powell modifications and the augmented forms of the studied systems are presented in \S\ref{sec21} and \S\ref{sec31},
the new numerical methods for the resulting augmented systems are introduced in \S\ref{sec22} and \S\ref{sec32}, and the results of the
preformed numerical experiments are reported in \S\ref{sec23} and \S\ref{sec33}. We conclude the paper and discuss some of our future
research plans in \S\ref{sec4}.

\section{Ideal MHD Equations}\label{sec2}
\subsection{Governing Equations}\label{sec21}
The ideal MHD equations read as
\allowdisplaybreaks
\begin{align}
&\rho_t+\nabla\cdot\p{\rho\bm u}=0,\nonumber\\
&(\rho\bm u)_t+\nabla\cdot\Big[\rho\bm u\bm u^\top+\Big(p+\hf\abs{\bm b}^2\Big)I-\bm b\bm b^\top\Big]=\bm0,\nonumber\\[-1ex]
&\label{2.1f}\\[-1.5ex]
&\bm b_t-\nabla\times(\bm u\times\bm b)=\bm0,\nonumber\\
&{\cal E}_t+\nabla\cdot\Big[\Big({\cal E}+p+\hf\abs{\bm b}^2\Big)\bm u-\bm b(\bm u\cdot\bm b)\Big]=0,\nonumber
\end{align}
where $t$ represents the time, $\rho$ is the density, $p$ is the pressure, $\bm u=(u,v,w)^\top$ represents the fluid velocity,
$\bm b=(b_1,b_2,b_3)^\top$ is the magnetic field, and ${\cal E}$ is the total energy. Additionally, $I$ is the $3\times3$ identity matrix,
$\gamma$ represents the ratio of specific heats, and finally, the system is completed through the equation of state (EOS)
\begin{equation}
{\cal E}=\frac{p}{\gamma-1}+\frac{\rho}{2}\abs{\bm u}^2+\hf\abs{\bm b}^2.
\label{2.6}
\end{equation}
where $|\cdot|$ represents the Euclidean norm. It is easy to show that
\begin{equation}
\nabla\cdot\bm b=0,
\label{2.5}
\end{equation} 
as long as initially the magnetic field is divergence-free.

As mentioned in \S\ref{sec1}, we follow a commonly-used approach and, instead of considering \eref{2.1f}, we develop a new numerical method
for the Godunov-Powell modified ideal MHD equations:
\begin{equation}
\begin{aligned}
&\rho_t+\nabla\cdot\p{\rho\bm u}=0,\\
&(\rho\bm u)_t+\nabla\cdot\Big[\rho\bm u\bm u^\top+\Big(p+\hf\abs{\bm b}^2\Big)I-\bm b\bm b^\top\Big]=-\bm b\p{\nabla\cdot\bm b},\\
&\bm b_t-\nabla\times(\bm u\times\bm b)=-\bm u\p{\nabla\cdot\bm b},\\
&{\cal E}_t+\nabla\cdot\Big[\Big({\cal E}+p+\hf\abs{\bm b}^2\Big)\bm u-\bm b(\bm u\cdot\bm b)\Big]=-\p{\bm u\cdot\bm b}\p{\nabla\cdot\bm b},
\end{aligned}
\label{2.4}
\end{equation}
which is closed with the help of the same EOS \eref{2.6}. Note that, theoretically, the Godunov-Powell source terms
$-\bm b\p{\nabla\cdot\bm b}$, $-\bm u\p{\nabla\cdot\bm b}$, and $-\p{\bm u\cdot\bm b}\p{\nabla\cdot\bm b}$ on the right hand side (RHS) of
\eref{2.4}, are zero due to the divergence-free condition \eref{2.5}. However, when numerically solving \eref{2.4} with EOS \eref{2.6},
these added relaxation terms help to reduce the divergence errors and enhance the robustness; see, e.g.,
\cite{powell1995upwind,powell1999solution,fuchs2011approximate,janhunen2000positive,waagan2011robust,WuShu2018,wu2019provably}. It is worth
noting that, although our proposed schemes are locally divergence-free, there are jumps of normal magnetic component across cell interfaces,
and the inclusion of these extra source terms can help to control the (weak) divergence errors at cell interfaces.

In this paper, we restrict our attention to the 2-D case, where all the quantities of interest depend on the spatial variables $x$ and
$y$ and time $t$ only. In this case, the divergence-free condition \eref{2.5} reads as $(b_1)_x+(b_2)_y=0$ and one of the goals in the
development of a good numerical method for the ideal MHD system \eref{2.4}, \eref{2.6} is to enforce this condition at the discrete level.
In order to achieve this goal, we introduce the new variables $A:=(b_1)_x$ and $B:=(b_2)_y$, and differentiate the induction equation for
$b_1$ and $b_2$ in \eref{2.4} with respect to $x$ and $y$, respectively, to obtain the following two evolution equations for $A$ and $B$:
\begin{equation}
\begin{aligned}
&A_t+\big(uA-b_2u_y\big)_x+\big(vA+b_1v_x\big)_y=0,\\
&B_t+\big(uB+b_2u_y\big)_x+\big(vB-b_1v_x\big)_y=0.
\end{aligned}
\label{2.8}
\end{equation}
From now on, we will add these equations to the Godunov-Powell modified ideal MHD equations and will numerically solve the augmented system
\eref{2.4}--\eref{2.8}, \eref{2.6}. Even though the number of equations to be discretized has been increased, adding the equations in
\eref{2.8} makes it easier to control the divergence-free constraint, which now reads as $A+B=0$.

Before introducing the numerical method, we write the augmented ideal MHD system \eref{2.4}--\eref{2.8} in the following vector form:
\begin{equation}
\bm U_t+\bm F(\bm U)_x+\bm G(\bm U)_y=Q(\bm U)\bm U_x+R(\bm U)\bm U_y,
\label{2.9}
\end{equation}
where $\bm U=(\rho,\rho u,\rho v,\rho w,b_1,b_2,b_3,\mc{E},A,B)^\top$,
\begin{align}
&\hspace*{-0.5cm}\bm F(\bm U)=
\begin{pmatrix}\rho u\\\rho u^2+p+\hf\abs{\bm b}^2-b_1^2\\\rho uv-b_1b_2\\\rho uw-b_1b_3\\0\\ub_2-vb_1\\ub_3-wb_1\\
\big({\cal E}+p+\hf\abs{\bm b}^2\big)u-(\bm u\cdot\bm b)b_1\\uA-b_2u_y\\uB+b_2u_y\end{pmatrix},~~
Q(\bm U)=\begin{pmatrix}0&0&0&0&0&0&0&0&0&0\\
0&0&0&0&-b_1&0&0&0&0&0\\
0&0&0&0&-b_2&0&0&0&0&0\\
0&0&0&0&-b_3&0&0&0&0&0\\
0&0&0&0&-u&0&0&0&0&0\\
0&0&0&0&-v&0&0&0&0&0\\
0&0&0&0&-w&0&0&0&0&0\\
0&0&0&0&-\bm u\cdot\bm b&0&0&0&0&0\\
0&0&0&0&0&0&0&0&0&0\\
0&0&0&0&0&0&0&0&0&0
\end{pmatrix},\nonumber\\
\label{2.10}\\
&\hspace*{-0.5cm}\bm G(\bm U)=
\begin{pmatrix}\rho v\\\rho uv-b_1b_2\\\rho v^2+p+\hf\abs{\bm b}^2-b_2^2\\\rho vw-b_2b_3\\vb_1-ub_2\\0\\vb_3-wb_2\\
\big({\cal E}+p+\hf\abs{\bm b}^2\big)v-(\bm u\cdot\bm b)b_2\\vA+b_1v_x\\vB-b_1v_x,\end{pmatrix},~~
R(\bm U)=\begin{pmatrix}0&0&0&0&0&0&0&0&0&0\\
0&0&0&0&0&-b_1&0&0&0&0\\
0&0&0&0&0&-b_2&0&0&0&0\\
0&0&0&0&0&-b_3&0&0&0&0\\
0&0&0&0&0&-u&0&0&0&0\\
0&0&0&0&0&-v&0&0&0&0\\
0&0&0&0&0&-w&0&0&0&0\\
0&0&0&0&0&-\bm u\cdot\bm b&0&0&0&0\\
0&0&0&0&0&0&0&0&0&0\\
0&0&0&0&0&0&0&0&0&0
\end{pmatrix}.\nonumber
\end{align}

\subsection{Numerical Method}\label{sec22}
We introduce a uniform Cartesian mesh consisting of the finite-volume cells $C_{j,k}=[x_\jmh,x_\jph]\times[y_\kmh,y_\kph]$ with
$x_\jph-x_\jmh\equiv\dx$ and $y_\kph-y_\kmh\equiv\dy$. We assume that at a certain time level $t$, the computed solution realized in terms
of its cell averages,
$$
\xbar{\bm U}_{j,k}\approx\frac{1}{\dx\dy}\iint\limits_{C_{j,k}}\bm U(x,y,t)\,{\rm d}x\,{\rm d}y,
$$
is available. Notice that the dependence of $\,\xbar{\bm U}_{j,k}$ and many other indexed quantities on $t$ is omitted here and throughout
the rest of the paper for the sake of brevity.

The cell averages $\,\xbar{\bm U}_{j,k}$ are evolved in time by implementing a dimension-by-dimension extension of the PCCU scheme from
\cite{CKM}, which results in the following semi-discretization of \eref{2.9}--\eref{2.10}:
\allowdisplaybreaks
\begin{align}
\frac{{\rm d}}{{\rm d}t}\,\xbar{\bm U}_{j,k}=&-\frac{1}{\dx}\bigg[\bm{{\cal F}}_{\jph,k}-\bm{{\cal F}}_{\jmh,k}-\bm Q_{j,k}\nonumber\\
&\hspace*{1.3cm}-\frac{s_{\jmh,k}^+}{s_{\jmh,k}^+-s_{\jmh,k}^-}\,\bm Q_{\bm\Psi,\jmh,k}
+\frac{s_{\jph,k}^-}{s_{\jph,k}^+-s_{\jph,k}^-}\,\bm Q_{\bm\Psi,\jph,k}\bigg]\nonumber\\[-1ex]
&\label{2.13}\\[-1.5ex]
&-\frac{1}{\dy}\bigg[\bm{{\cal G}}_{j,\kph}-\bm{{\cal G}}_{j,\kmh}-\bm R_{j,k}\nonumber\\
&\hspace*{1.3cm}-\frac{s_{j,\kmh}^+}{s_{j,\kmh}^+-s_{j,\kmh}^-}\,\bm R_{\bm\Psi,j,\kmh}
+\frac{s_{j,\kph}^-}{s_{j,\kph}^+-s_{j,\kph}^-}\,\bm R_{\bm\Psi,j,\kph}\bigg].\nonumber
\end{align}
Here,
\begin{equation}
\begin{aligned}
\bm{{\cal F}}_{\jph,k}&=\frac{s_{\jph,k}^+\bm F\big(\bm U^{\rm E}_{j,k}\big)-s_{\jph,k}^-\bm F\big(\bm U^{\rm W}_{j+1,k}\big)}
{s_{\jph,k}^+-s_{\jph,k}^-}+
\frac{s_{\jph,k}^+s_{\jph,k}^-}{s_{\jph,k}^+-s_{\jph,k}^-}\left(\bm U^{\rm W}_{j+1,k}-\bm U^{\rm E}_{j,k}\right),\\[0.5ex]
\bm{{\cal G}}_{j,\kph}&=\frac{s_{j,\kph}^+\bm G\big(\bm U^{\rm N}_{j,k}\big)-s_{j,\kph}^-\bm G\big(\bm U^{\rm S}_{j,k+1}\big)}
{s_{j,\kph}^+-s_{j,\kph}^-}+
\frac{s_{j,\kph}^+s_{j,\kph}^-}{s_{j,\kph}^+-s_{j,\kph}^-}\left(\bm U^{\rm S}_{j,k+1}-\bm U^{\rm N}_{j,k}\right)
\end{aligned}
\label{2.18}
\end{equation}
are the CU numerical fluxes from \cite{KTrp}, $\bm U^{\rm E,\rm W,\rm N,\rm S}_{j,k}$ are the reconstructed point values at the cell
interfaces of cell $C_{j,k}$ (see \S\ref{sec221} for details), $s_{\jph,k}^\pm$ and $s_{j,\kph}^\pm$ are the one-sided local speeds of
propagation in the $x$- and $y$-directions, respectively (see \S\ref{sec222} for details), and $\bm Q_{j,k}$, $\bm R_{j,k}$,
$\bm Q_{\bm\Psi,\jph,k}$, and $\bm R_{\bm\Psi,j,\kph}$ denote the discretizations of the nonconservative products on the RHS of \eref{2.9}
(see \S\ref{sec223} for details).

We point out that \eref{2.13} is a system of ODEs, which should be numerically integrated in time by an appropriate ODE solver. In the
numerical experiments reported in \S\ref{sec23}, we have used the explicit three-stage third-order strong stability preserving (SSP)
Runge-Kutta method; see, e.g., \cite{Gottlieb2001Strong,Gottlieb2011Strong}.

\subsubsection{Piecewise Linear Reconstruction}\label{sec221}
Equipped with the cell averages $\,\xbar{\bm U}_{j,k}$, we first use the EOS \eref{2.6} and compute the approximate point values of $u$,
$v$, $w$, and $p$ at the cell centers:
$$
\begin{aligned}
&u_{j,k}=\frac{(\xbar{\rho u})_{j,k}}{\xbar\rho_{j,k}},\quad v_{j,k}=\frac{(\xbar{\rho v})_{j,k}}{\xbar\rho_{j,k}},\quad
w_{j,k}=\frac{(\xbar{\rho w})_{j,k}}{\xbar\rho_{j,k}},\\
&p_{j,k}=\p{\gamma-1}\bigg[\,\xbar{\cal E}_{j,k}-\hf\,\xbar\rho_{j,k}\left(u_{j,k}^2+v_{j,k}^2+w_{j,k}^2\right)-
\hf\Big((\xbar{b_1})_{j,k}^2+(\xbar{b_2})_{j,k}^2+(\xbar{b_3})_{j,k}^2\Big)\bigg].
\end{aligned}
$$
We then introduce a new set of discrete variables.
$$
\bm W_{j,k}:=(\,\xbar\rho_{j,k},u_{j,k},v_{j,k},w_{j,k},(\xbar{b_1})_{j,k},(\xbar{b_2})_{j,k},(\xbar{b_3})_{j,k},p_{j,k},\xbar A_{j,k},
\xbar B_{j,k})^\top,
$$
and compute the cell interface point values $\bm W^{\rm E,\rm W,\rm N,\rm S}_{j,k}$ using a proper conservative piecewise linear
reconstruction
\begin{equation}
\widetilde{\bm W}(x,y)=\bm W_{j,k}+(\bm W_x)_{j,k}(x-x_j)+(\bm W_y)_{j,k}(y-y_k),\quad(x,y)\in C_{j,k},
\label{2.14}
\end{equation}
which results in
\begin{equation}
\begin{aligned}
&\bm W^{\rm E}_{j,k}=\bm W_{j,k}+(\bm W_x)_{j,k}\frac{\dx}{2},\quad\bm W^{\rm W}_{j,k}=\bm W_{j,k}-(\bm W_x)_{j,k}\frac{\dx}{2},\\
&\bm W^{\rm N}_{j,k}=\bm W_{j,k}+(\bm W_y)_{j,k}\frac{\dy}{2},\quad\bm W^{\rm S}_{j,k}=\bm W_{j,k}-(\bm W_y)_{j,k}\frac{\dy}{2}.
\end{aligned}
\label{2.15}
\end{equation}

In order for \eref{2.14} to be second-order accurate, the slopes $(\bm W_x)_{j,k}$ and $(\bm W_y)_{j,k}$ have to be at least first-order
approximations of $\bm W_x(x_j,y_k)$ and $\bm W_y(x_j,y_k)$, respectively. A non-oscillatory nature of the piecewise linear reconstruction
\eref{2.14} is typically ensured with the help of a nonlinear limiter. To all of the components of $\bm W$, we compute the slopes (except for
$((b_1)_x)_{j,k}$ and $((b_2)_y)_{j,k}$) using the generalized minmod limiter (see, e.g.,
\cite{Lie2003artificial,Nessyahu1990Non,Sweby1984High}):
\begin{equation}
\begin{aligned}
&(W^{(i)}_x)_{j,k}={\rm minmod}\left(\theta\,\frac{W^{(i)}_{j,k}-W^{(i)}_{j-1,k}}{\dx},\,\frac{W^{(i)}_{j+1,k}-W^{(i)}_{j-1,k}}{2\dx},\,
\theta\,\frac{W^{(i)}_{j+1,k}-W^{(i)}_{j,k}}{\dx}\right),\quad i\ne5,\\
&(W^{(i)}_y)_{j,k}={\rm minmod}\left(\theta\,\frac{W^{(i)}_{j,k}-W^{(i)}_{j,k-1}}{\dy},\,\frac{W^{(i)}_{j,k+1}-W^{(i)}_{j,k-1}}{2\dy},\,
\theta\,\frac{W^{(i)}_{j,k+1}-W^{(i)}_{j,k}}{\dy}\right),\quad i\ne6,
\end{aligned}
\label{2.16}
\end{equation}
where the minmod function is defined by
\begin{equation}
{\rm minmod}(a_1,a_2,\dots)= 
\begin{cases}
{\rm sgn}(a_k)\min\{\abs{a_1},\abs{a_2},\dots\}&\mbox{if}~{\rm sgn}(a_k)=1~\mbox{or}~{\rm sgn}(a_k)=-1~\forall k,\\
0&\mbox{otherwise}.
\end{cases}
\label{2.17f}
\end{equation}

The slopes $((b_1)_x)_{j,k}$ and $((b_2)_y)_{j,k}$, however, should not be computed using the generalized minmod limiter or any other
conventional limiter as our goal is to enforce local divergence-free condition \eref{2.5}, which at the discrete level reads as
$((b_1)_x)_{j,k}+((b_2)_y)_{j,k}\equiv0$ for all $j,k$. This goal can be achieved if we set 
\begin{equation}
((b_1)_x)_{j,k}=\,\xbar A_{j,k}\quad\mbox{and}\quad((b_1)_y)_{j,k}=\,\xbar B_{j,k},
\label{2.18f}
\end{equation}
since
\begin{equation}
\xbar A_{j,k}+\,\xbar B_{j,k}=0
\label{2.16f}
\end{equation}
is true for all $j,k$ provided \eref{2.16f} is satisfied at time $t=0$; see Theorem \ref{th21} in \S\ref{sec224}. 

While the use of \eref{2.18f} guarantees the local discrete divergence-free condition, the resulting reconstruction of $b_1$ and $b_2$ may
be oscillatory in the $x$- and $y$-directions, respectively. As we have observed in several numerical experiments, this often leads
to an oscillatory numerical solution. We, therefore, adjust the slopes in \eref{2.18f} by scaling them as follows. 

We begin by introducing the auxiliary slopes $(\widehat{(b_1)}_x)_{j,k}$ and $(\widehat{(b_2)}_y)_{j,k}$, which are computed using the
aforementioned generalized minmod reconstruction. The reconstructions of $b_1$ and $b_2$ can then be made both non-oscillatory and 
locally divergence-free by replacing \eref{2.18f} with
\begin{equation}
((b_1)_x)_{j,k}=\sigma_{j,k}\,\xbar A_{j,k},\quad((b_1)_y)_{j,k}=\sigma_{j,k}\,\xbar B_{j,k},
\label{2.20f}
\end{equation}
where 
\begin{equation}
\sigma_{j,k}=\min\Big\{1,\sigma^x_{j,k},\sigma^y_{j,k}\Big\}
\label{2.21f}
\end{equation}
and the scaling factors $\sigma^x_{j,k}$ and $\sigma^y_{j,k}$ are computed by
\begin{equation}
\sigma^x_{j,k}:=\left\{\begin{aligned}
&\min\bigg\{1,\frac{(\widehat{(b_1)}_x)_{j,k}}{\xbar A_{j,k}}\bigg\}&&\mbox{if}~(\widehat{(b_1)}_x)_{j,k}\,\xbar A_{j,k}>0,\\
&0&&\mbox{otherwise},
\end{aligned}\right.
\label{2.22f}
\end{equation}
and 
\begin{equation}
\sigma^y_{j,k}:=\left\{\begin{aligned}
&\min\bigg\{1,\frac{(\widehat{(b_2)}_y)_{j,k}}{\xbar B_{j,k}}\bigg\}&&\mbox{if}~(\widehat{(b_2)}_y)_{j,k}\,\xbar B_{j,k}>0,\\
&0&&\mbox{otherwise}.
\end{aligned}\right.
\label{2.23f}
\end{equation}

Finally, equipped with \eref{2.15}, we use the EOS \eref{2.6} to compute the cell interface point values
${\cal E}^{\rm E,\rm W,\rm N,\rm S}_{j,k}$ as follows:
\begin{equation}
{\cal E}_{j,k}^\ell=\frac{p_{j,k}^\ell}{\gamma-1}+\hf\rho^\ell_{j,k}\left[(u^\ell_{j,k})^2+(v^\ell_{j,k})^2+(w^\ell_{j,k})^2\right]+
\hf\left[\p{(b_1)^\ell_{j,k}}^2+\p{(b_2)^\ell_{j,k}}^2+\p{(b_3)^\ell_{j,k}}^2\right],
\label{2.20}
\end{equation}
where $\ell\in\{{\rm E},{\rm W},{\rm N},{\rm S}\}$.
\begin{rmk}\label{rem21}
We note that we have reconstructed the primitive variables $u$, $v$, $w$, and $p$ rather than the conservative variables $\rho u$, $\rho v$,
$\rho w$, and ${\cal E}$ since our numerical experiments clearly indicate that the resulting scheme, which is based on the reconstruction of
the primitive variables, is less oscillatory and produces no negative pressure values.
\end{rmk}

It is important to point out that the ninth and tenth components of the fluxes $\bm F(\bm U^{\rm E(W)}_{j,k})$ and
$\bm G(\bm U^{\rm N(S)}_{j,k})$ depend not only on the corresponding point values of $u$, $v$, $A$, $B$, $b_1$, and $b_2$, but also on the
point values of the derivatives $(u_y)^{\rm E(W)}_{j,k}$ and $(v_x)^{\rm N(S)}_{j,k}$. We compute these values using the first-order
approximation, namely, we set
$$
(u_y)^{\rm E}_{j,k}=(u_y)_{j,k},\quad(u_y)^{\rm W}_{j,k}=(u_y)_{j,k},\quad(v_x)^{\rm N}_{j,k}=(v_x)_{j,k},\quad
(v_x)^{\rm S}_{j,k}=(v_x)_{j,k},
$$
where the slopes $(u_y)_{j,k}$ and $(v_x)_{j,k}$ are computed by \eref{2.16}. Notice that even though this will result in the first-order
approximation of the auxiliary variables $A$ and $B$, the other components of $\bm U$ will be still computed with the second order and thus
the second-order accuracy of the resulting scheme will not be affected.

\subsubsection{One-Sided Speeds of Propagation}\label{sec222}
Equipped with the reconstructed point values \eref{2.15} and \eref{2.20}, we now proceed with the computation of the one-sided local speeds
of propagation $s^\pm_{\jph,k}$ and $s^\pm_{j,\kph}$ seen in \eref{2.13} and \eref{2.18}. We stress that when the PCCU schemes are applied
to general nonconservative systems of type \eref{2.9}, the $x$- and $y$-directional speeds would typically be estimated using the largest
and smallest eigenvalues of the matrices $\frac{\partial\bm F}{\partial\bm U}(\bm U)-Q(\bm U)$ and
$\frac{\partial\bm G}{\partial\bm U}(\bm U)-R(\bm U)$, respectively. However, it is known (see, e.g., \cite{einfeldt1991godunov}) that in
the context of the ideal MHD system \eref{2.9}--\eref{2.10}, the estimates, which are solely based on the eigenvalues mentioned above may be
inaccurate and using them may lead to severe instabilities.

We, therefore, follow \cite{wu2019provably}, where the propagation speeds were slightly overestimated to ensure the positivity of both the
computed density and pressure, and estimate the right- and left-sided local speeds in the $x$-direction by
\begin{equation*}
\begin{aligned}
&s_{\jph,k}^+=\max\Big\{\max\big\{u_{j,k}^{\rm E},u_{\jph,k}^{\rm Roe}\big\}+c_{j,k}^{\rm E}+\beta^x_{\jph,k},\,
\max\big\{u_{j+1,k}^{\rm W},u_{\jph,k}^{\rm Roe}\big\}+c_{j+1,k}^{\rm W}+\beta^x_{\jph,k},\,0\Big\},\\
&s_{\jph,k}^-=\min\Big\{\min\big\{u_{j,k}^{\rm E},u_{\jph,k}^{\rm Roe}\big\}-c_{j,k}^{\rm E}-\beta^x_{\jph,k},\,
\min\big\{u_{j+1,k}^{\rm W},u_{\jph,k}^{\rm Roe}\big\}-c_{j+1,k}^{\rm W}-\beta^x_{\jph,k},\,0\Big\},
\end{aligned}
\end{equation*}
where 
\begin{equation*}
u_{\jph,k}^{\rm Roe}:=\frac{u_{j,k}^{\rm E}\sqrt{\rho_{j,k}^{\rm E}}+u_{j+1,k}^{\rm W}\sqrt{\rho_{j+1,k}^{\rm W}}}
{\sqrt{\rho_{j,k}^{\rm E}}+\sqrt{\rho_{j+1,k}^{\rm W}}},\quad
\beta^x_{\jph,k}:=\frac{\abs{\bm b_{j,k}^{\rm E}-\bm b_{j+1,k}^{\rm W}}}{\sqrt{\rho_{j,k}^{\rm E}}+\sqrt{\rho_{j+1,k}^{\rm W}}},
\end{equation*}
and $c^{\rm E(W)}_{j,k}$ are the fast magneto-acoustic wave speeds computed using
\begin{equation*}
\left(c_{j,k}^{\rm E(W)}\right)^2=\frac{1}{2\rho_{j,k}^{\rm E(W)}}\br{\gamma p_{j,k}^{\rm E(W)}+\abs{\bm b_{j,k}^{\rm E(W)}}^2+
\sqrt{\p{\gamma p_{j,k}^{\rm E(W)}+\abs{\bm b_{j,k}^{\rm E(W)}}^2}^2-4\gamma p_{j,k}^{\rm E(W)}\p{\p{b_1}_{j,k}^{\rm E(W)}}^2}~}.
\end{equation*}

Similarly, we estimate the corresponding $y$-directional speeds by
\begin{equation*}
\begin{aligned}
&s_{j,\kph}^+=\max\Big\{\max\big\{v_{j,k}^{\rm N},v_{j,\kph}^{\rm Roe}\big\}+c_{j,k}^{\rm N}+\beta^y_{j,\kph},\,
\max\big\{v_{j,k+1}^{\rm S},v_{j,\kph}^{\rm Roe}\big\}+c_{j,k+1}^{\rm S}+\beta^y_{j,\kph},\,0\Big\},\\
&s_{j,\kph}^-=\min\Big\{\min\big\{v_{j,k}^{\rm N},v_{j,\kph}^{\rm Roe}\big\}-c_{j,k}^{\rm N}-\beta^y_{j,\kph},\,
\min\big\{v_{j,k+1}^{\rm S},v_{j,\kph}^{\rm Roe}\big\}-c_{j,k+1}^{\rm S}-\beta^y_{j,\kph},\,0\Big\},
\end{aligned}
\end{equation*}
where 
\begin{equation*}
v_{j,\kph}^{\rm Roe}:=\frac{v_{j,k}^{\rm N}\sqrt{\rho_{j,k}^{\rm N}}+v_{j,k+1}^{\rm S}\sqrt{\rho_{j,k+1}^{\rm S}}}
{\sqrt{\rho_{j,k}^{\rm N}}+\sqrt{\rho_{j,k+1}^{\rm S}}},\quad
\beta^y_{j,\kph}:=\frac{\abs{\bm b_{j,k}^{\rm N}-\bm b_{j,k+1}^{\rm S}}}{\sqrt{\rho_{j,k}^{\rm N}}+\sqrt{\rho_{j,k+1}^{\rm S}}},
\end{equation*}
\begin{equation*}
\left(c_{j,k}^{\rm N(S)}\right)^2=\frac{1}{2\rho_{j,k}^{\rm N(S)}}\br{\gamma p_{j,k}^{\rm N(S)}+\abs{\bm b_{j,k}^{\rm N(S)}}^2+
\sqrt{\p{\gamma p_{j,k}^{\rm N(S)}+\abs{\bm b_{j,k}^{\rm N(S)}}^2}^2-4\gamma p_{j,k}^{\rm N(S)}\p{\p{b_1}_{j,k}^{\rm N(S)}}^2}~}.
\end{equation*}

\subsubsection{Discretization of the Nonconservative Products}\label{sec223}
In this section, we provide the computation of the nonconservative product terms in \eref{2.13}.

Following \cite{CKM} (see also \cite{CKN22}), we obtain nonconservative terms in the $x$-direction, $\bm Q_{j,k}$ and
$\bm Q_{\bm\Psi,\jph,k}$, as follows. First, in order to compute the term $\bm Q_{j,k}$, we take a global (in space) interpolant
$\bm U(\widetilde{\bm W}(x,y))$, where $\widetilde{\bm W}$ is given by \eref{2.14}, and evaluate the integral in
\begin{equation*}
\bm Q_{j,k}=\int\limits_{x_\jmh}^{x_\jph}Q\big(\bm U\big(\widetilde{\bm W}(x,y_k)\big)\big)\bm U(\widetilde{\bm W}(x,y_k))_x\,{\rm d}x,
\end{equation*}
where $Q(\bm U)$ is defined in \eref{2.10}, exactly. This results in the following expressions for the ten components of the vector 
$\bm Q_{j,k}$:
\allowdisplaybreaks
\begin{align*}
&Q^{(1)}_{j,k}=Q^{(9)}_{j,k}=Q^{(10)}_{j,k}=0,\\
&Q^{(i)}_{j,k}=-\int\limits_{x_\jmh}^{x_\jph}\widetilde{b_{i-1}}(x,y_k)((b_1)_x)_{j,k}\,{\rm d}x=
-(\,\xbar{b_{i-1}})_{j,k}\sigma_{j,k}\,\xbar A_{j,k}\dx,\quad i=2,3,4,\\
&\big(Q^{(5)}_{j,k},Q^{(6)}_{j,k},Q^{(7)}_{j,k}\big)^\top=-\int\limits_{x_\jmh}^{x_\jph}\widetilde{\bm u}(x,y_k)((b_1)_x)_{j,k}\,{\rm d}x=
-\bm u_{j,k}\sigma_{j,k}\,\xbar A_{j,k}\dx,\\
&Q^{(8)}_{j,k}=-\int\limits_{x_\jmh}^{x_\jph}\widetilde{\bm u}(x,y_k)\cdot\widetilde{\bm b}(x,y_k)((b_1)_x)_{j,k}\,{\rm d}x=
-\Big[\Big(\bm u_{j,k}\cdot\bm b_{j,k}\\
&\hspace*{1.8cm}+\frac{(\dx)^2}{12}
\left\{(u_x)_{j,k}\sigma_{j,k}\,\xbar A_{j,k}+(v_x)_{j,k}((b_2)_x)_{j,k}+(w_x)_{j,k}((b_3)_x)_{j,k}\right\}
\Big)\Big]\sigma_{j,k}\,\xbar A_{j,k}\dx,
\end{align*}
where we have used the slopes $((b_1)_x)_{j,k}$ given by \eref{2.20f}--\eref{2.23f}, while the other slopes are computed in
\eref{2.16}--\eref{2.17f}.

Next, the terms $\bm Q_{\bm\Psi,\jph,k}$ are computed by the exact integration of
\begin{equation*}
\bm Q_{\bm\Psi,\jph,k}=\int\limits_0^1Q\big(\bm U\big(\bm\Psi_{\jph,k}(s)\big)\big)\bm\Psi_{\jph,k}'(s)\,{\rm d}s,
\end{equation*}
where $\bm\Psi_{\jph,k}(s)$ is a linear path connecting the states $\bm W^{\rm E}_{j,k}$ and $\bm W^{\rm W}_{j+1,k}$:
\begin{equation*}
\bm\Psi_{\jph,k}(s)=\bm W^{\rm E}_{j,k}+s\p{\bm W^{\rm W}_{j+1,k}-\bm W^{\rm E}_{j,k}}.
\end{equation*}
This results in
\allowdisplaybreaks
\begin{align*}
&Q^{(1)}_{\bm\Psi,\jph,k}=Q^{(9)}_{\bm\Psi,\jph,k}=Q^{(10)}_{\bm\Psi,\jph,k}=0,\\
&Q^{(i)}_{\bm\Psi,\jph,k}=-\int\limits_0^1\left\{(b_{i-1})^{\rm E}_{j,k}+s\p{(b_{i-1})^{\rm W}_{j+1,k}-(b_{i-1})^{\rm E}_{j,k}}\right\}
[b_1]_{\jph,k}\,{\rm d}s\\
&\hspace*{1.55cm}=-\hf\left((b_{i-1})^{\rm E}_{j,k}+(b_{i-1})^{\rm W}_{j+1,k}\right)[b_1]_{\jph,k},\quad i=2,3,4,\\
&\big(Q^{(5)}_{\bm\Psi,\jph,k},Q^{(6)}_{\bm\Psi,\jph,k},Q^{(7)}_{\bm\Psi,\jph,k}\big)^\top=
-\int\limits_0^1\left\{\bm u^{\rm E}_{j,k}+s\p{\bm u^{\rm W}_{j+1,k}-\bm u^{\rm E}_{j,k}}\right\}[b_1]_{\jph,k}\,{\rm d}s\\
&\hspace*{5.6cm}=-\hf\left(\bm u^{\rm E}_{j,k}+\bm u^{\rm W}_{j+1,k}\right)[b_1]_{\jph,k},\\
&Q^{(8)}_{\bm\Psi,\jph,k}=-\int\limits_0^1\left\{\bm u^{\rm E}_{j,k}+s\p{\bm u^{\rm W}_{j+1,k}-\bm u^{\rm E}_{j,k}}\right\}\cdot
\left\{\bm b^{\rm E}_{j,k}+s\p{\bm b^{\rm W}_{j+1,k}-\bm b^{\rm E}_{j,k}}\right\}[b_1]_{\jph,k}\,{\rm d}s\\
&\hspace*{1.55cm}=-\frac{1}{6}\Big(2\bm u_{j,k}^{\rm E}\cdot\bm b_{j,k}^{\rm E}+\bm u_{j,k}^{\rm E}\cdot\bm b_{j+1,k}^{\rm W}+
\bm u_{j+1,k}^{\rm W}\cdot\bm b_{j,k}^{\rm E}+2\bm u_{j+1,k}^{\rm W}\cdot\bm b_{j+1,k}^{\rm W}\Big)[b_1]_{\jph,k},
\end{align*}
where $[b_1]_{\jph,k}:=(b_1)^{\rm W}_{j+1,k}-(b_1)^{\rm E}_{j,k}$.

Similarly, we obtain the following formulae for the nonconservative terms in the $y$-direction, $\bm R_{j,k}$ and $\bm R_{\bm\Psi,j,\kph}$:
\allowdisplaybreaks
\begin{align*}
&R^{(1)}_{j,k}=R^{(9)}_{j,k}=R^{(10)}_{j,k}=0,\quad R^{(i)}_{j,k}=-(\,\xbar{b_{i-1}})_{j,k}\sigma_{j,k}\,\xbar B_{j,k}\dy,\quad i=2,3,4,\\
&\big(R^{(5)}_{j,k},R^{(6)}_{j,k},R^{(7)}_{j,k}\big)^\top=-\bm u_{j,k}\sigma_{j,k}\,\xbar B_{j,k}\dy,\\
&R^{(8)}_{j,k}=-\Big[\Big(\bm u_{j,k}\cdot\bm b_{j,k}
+\frac{(\dy)^2}{12}\left\{(u_y)_{j,k}((b_1)_y)_{j,k}+(v_y)_{j,k}\sigma_{j,k}\,\xbar B_{j,k}+(w_y)_{j,k}((b_3)_y)_{j,k}\right\}
\Big)\Big]\sigma_{j,k}\,\xbar B_{j,k}\dy,
\end{align*}
where we have used the slopes $((b_2)_y)_{j,k}$ given by \eref{2.20f}--\eref{2.23f}, while the other slopes are computed in
\eref{2.16}--\eref{2.17f}, and
\allowdisplaybreaks
\begin{align*}
&R^{(1)}_{\bm\Psi,j,\kph}=R^{(9)}_{\bm\Psi,j,\kph}=R^{(10)}_{\bm\Psi,j,\kph}=0,\quad
R^{(i)}_{\bm\Psi,j,\kph}=-\hf\left((b_{i-1})^{\rm N}_{j,k}+(b_{i-1})^{\rm S}_{j,k+1}\right)[b_2]_{j,\kph},\quad i=2,3,4,\\
&\big(R^{(5)}_{\bm\Psi,j,\kph},R^{(6)}_{\bm\Psi,j,\kph},R^{(7)}_{\bm\Psi,j,\kph}\big)^\top=
-\hf\left(\bm u^{\rm N}_{j,k}+\bm u^{\rm S}_{j,k+1}\right)[b_2]_{j,\kph},\\
&R^{(8)}_{\bm\Psi,j,\kph}=-\frac{1}{6}\Big(2\bm u_{j,k}^{\rm N}\cdot\bm b_{j,k}^{\rm N}+\bm u_{j,k}^{\rm N}\cdot\bm b_{j,k+1}^{\rm S}+
\bm u_{j,k+1}^{\rm S}\cdot\bm b_{j,k}^{\rm N}+2\bm u_{j,k+1}^{\rm S}\cdot\bm b_{j,k+1}^{\rm S}\Big)[b_2]_{j,\kph},
\end{align*}
where $[b_2]_{j,\kph}:=\p{(b_2)^{\rm S}_{j,k+1}-(b_2)^{\rm N}_{j,k}}$.

\subsubsection{Local Divergence-Free Property}\label{sec224}
We now prove the local divergence-free property of the proposed PCCU scheme.
\begin{thm}\label{th21}
For the PCCU scheme \eref{2.13}--\eref{2.18} with the reconstruction described in \S\ref{sec221}, the local divergence-free condition
\begin{equation}
\big((b_1)_x\big)_{j,k}+\big((b_2)_y\big)_{j,k}=0,
\label{2.28f}
\end{equation}
holds for all $j,k$ and at all times, provided it is satisfied initially.
\end{thm}
\begin{proof}
First, we note that according to \eref{2.20f}, 
$$
\big((b_1)_x\big)_{j,k}+\big((b_2)_y\big)_{j,k}=\sigma_{j,k}\left(\,\xbar A_{j,k}+\,\xbar B_{j,k}\right).
$$
Therefore, in order to prove \eref{2.28f}, it is sufficient to show that $\xbar A_{j,k}+\,\xbar B_{j,k}=0$ for all $j,k$ and for all times
assuming that it is satisfied at the initial time $t=0$. 

We then observe that the quantities $\xbar A_{j,k}$ and $\xbar B_{j,k}$ are the ninth and tenth components of $\xbar{\bm U}_{j,k}$ and thus
they are evolved in time by numerically integrating the ninth and tenth components of \eref{2.13}--\eref{2.18}. Adding these components in
\eref{2.13} results in
\begin{equation}
\begin{aligned}
\frac{{\rm d}}{{\rm d}t}\left(\,\xbar A_{j,k}+\,\xbar B_{j,k}\right)=&-\frac{1}{\dx}
\bigg[{{\cal F}}^{(9)}_{\jph,k}-{{\cal F}}^{(9)}_{\jmh,k}+{{\cal F}}^{(10)}_{\jph,k}-{{\cal F}}^{(10)}_{\jmh,k}\bigg]\\
&-\frac{1}{\dy}\bigg[{{\cal G}}^{(9)}_{j,\kph}-{{\cal G}}^{(9)}_{j,\kmh}+{{\cal G}}^{(10)}_{j,\kph}-{{\cal G}}^{(10)}_{j,\kmh}\bigg].
\end{aligned}
\label{2.29}
\end{equation}
In order to complete the proof, it is sufficient to show that the RHS of \eref{2.29} vanishes as long as $\xbar A_{j,k}+\,\xbar B_{j,k}=0$
for all $j,k$. To this end, we use \eref{2.18} to evaluate 
$$
\begin{aligned}
{{\cal F}}^{(9)}_{\jph,k}-{{\cal F}}^{(9)}_{\jmh,k}+{{\cal F}}^{(10)}_{\jph,k}-{{\cal F}}^{(10)}_{\jmh,k}&=
\frac{s_{\jph,k}^+\left[u^{\rm E}_{j,k}\left(A_{j,k}^{\rm E}+B_{j,k}^{\rm E}\right)\right]
-s_{\jph,k}^-\left[u^{\rm W}_{j+1,k}\left(A_{j+1,k}^{\rm W}+B_{j+1,k}^{\rm W}\right)\right]}{s_{\jph,k}^+-s_{\jph,k}^-}\\
&+\frac{s_{\jph,k}^+s_{\jph,k}^-}{s_{\jph,k}^+-s_{\jph,k}^-}\Big[\left(A^{\rm W}_{j+1,k}+B^{\rm W}_{j+1,k}\right)
-\left(A^{\rm E}_{j,k}+B^{\rm E}_{j,k}\right)\Big]\\
\stackrel{\eref{2.15}}{=}\frac{s_{\jph,k}^+}{s_{\jph,k}^+-s_{\jph,k}^-}&
\left[u^{\rm E}_{j,k}\Big(\,\xbar A_{j,k}+\frac{\dx}{2}(A_x)_{j,k}+\,\xbar B_{j,k}+\frac{\dx}{2}(B_x)_{j,k}\Big)\right]\\
-\frac{s_{\jph,k}^-}{s_{\jph,k}^+-s_{\jph,k}^-}&
\left[u^{\rm W}_{j+1,k}\Big(\,\xbar A_{j+1,k}-\frac{\dx}{2}(A_x)_{j+1,k}+\,\xbar B_{j+1,k}-\frac{\dx}{2}(B_x)_{j+1,k}\Big)\right]\\
+\frac{s_{\jph,k}^+s_{\jph,k}^-}{s_{\jph,k}^+-s_{\jph,k}^-}&\left[\,\xbar A_{j+1,k}-\frac{\dx}{2}(A_x)_{j+1,k}+
\,\xbar B_{j+1,k}-\frac{\dx}{2}(B_x)_{j+1,k}\right]\\
-\frac{s_{\jph,k}^+s_{\jph,k}^-}{s_{\jph,k}^+-s_{\jph,k}^-}&\left[\,\xbar A_{j,k}+\frac{\dx}{2}(A_x)_{j,k}+
\,\xbar B_{j,k}+\frac{\dx}{2}(B_x)_{j,k}\right]\\
&\hspace*{-3.5cm}=\frac{\dx}{2}\Bigg\{\frac{s_{\jph,k}^+\left[u^{\rm E}_{j,k}\left((A_x)_{j,k}+(B_x)_{j,k}\right)\right]
+s_{\jph,k}^-\left[u^{\rm W}_{j+1,k}\left((A_x)_{j+1,k}+(B_x)_{j+1,k}\right)\right]}{s_{\jph,k}^+-s_{\jph,k}^-}\\
&\hspace*{-2.4cm}-\frac{s_{\jph,k}^+s_{\jph,k}^-}{s_{\jph,k}^+-s_{\jph,k}^-}\Big[(A_x)_{j+1,k}+(B_x)_{j+1,k}+(A_x)_{j,k}+(B_x)_{j,k}\Big]
\Bigg\},
\end{aligned}
$$
where the last equality is obtained using $\xbar A_{j,k}+\,\xbar B_{j,k}=\xbar A_{j+1,k}+\,\xbar B_{j+1,k}=0$.

It is now clear that ${{\cal F}}^{(9)}_{\jph,k}-{{\cal F}}^{(9)}_{\jmh,k}+{{\cal F}}^{(10)}_{\jph,k}-{{\cal F}}^{(10)}_{\jmh,k}$ will be
identically zero as long as
\begin{equation}
(A_x)_{j,k}+(B_x)_{j,k}=0
\label{2.30}
\end{equation}
for all $j,k$. Indeed, \eref{2.30} is true since $\,\xbar A_{j,k}=-\,\xbar B_{j,k}$ for all $j,k$ and both the slopes $(A_x)_{j,k}$ and
$(B_x)_{j,k}$ are computed using the same limiter \eref{2.16}.

Similarly, one can show that
${{\cal G}}^{(9)}_{j,\kph}-{{\cal G}}^{(9)}_{j,\kmh}+{{\cal G}}^{(10)}_{j,\kph}-{{\cal G}}^{(10)}_{j,\kmh}\equiv0$, so that the RHS of
\eref{2.29} vanishes and thus the proof of the theorem is complete.
\end{proof}

\subsection{Numerical Examples}\label{sec23}
In this section, we demonstrate the performance of the proposed PCCU scheme in several numerical experiments conducted for the augmented 2-D
ideal MHD system \eref{2.9}-\eref{2.10}, \eref{2.6}. In all of the examples in this section, we take the CFL number 0.25 and the minmod
parameter $\theta=1.3$. 

\paragraph{Example 1---Brio-Wu Shock-Tube Problem.} In the first example, we consider the one-dimensional (1-D) Riemann problem known as the
Brio-Wu shock tube problem, originally presented in \cite{brio1988upwind}. This problem is the standard test for capturing compound waves
that emerge as solutions of the ideal MHD system. We take the following initial data, which depend on $x$ only:
$$ 
\left(\rho,u,v,w,b_1,b_2,b_3,p\right)(x,y,0)=
\left\{\begin{aligned}
&(1,0,0,0,0.75,1,0,1),&&x<0,\\
&(0.125,0,0,0,0.75,-1,0,0.1),&&\mbox{otherwise},
\end{aligned}\right.
$$
and set the free boundary conditions on all sides of the computational domain $[-1,1]\times[-0.01,0.01]$. The specific heat ratio
$\gamma=2$ in this example.

In Figure \ref{fig1}, we plot the $y=0$ cross-section of the density $\rho$, $x$-magnetic field $b_1$, and $y$-magnetic field $b_2$ computed
on $800\times8$ and $1600\times16$ uniform meshes at time $t=0.2$. As one can see, the solution to this 1-D Riemann problem consists of
several nonsmooth structures, such as rarefaction waves, shocks traveling at various speeds, a contact discontinuity, and a compound shock
wave. The proposed PCCU scheme captures all of these complicated structures well, and the obtained results strongly agree with those
reported in \cite{dumbser2019divergence,fuchs2011approximate,li2005locally,li2011central,liu2018entropy}.
\begin{figure}[ht!]
\centerline{\includegraphics[trim=0.4cm 0.8cm 1.0cm 0.2cm, clip, width=5.2cm]{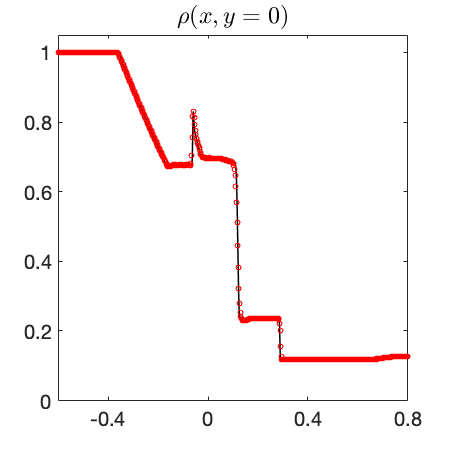}\hspace*{0.5cm}
            \includegraphics[trim=0.4cm 0.8cm 1.0cm 0.2cm, clip, width=5.2cm]{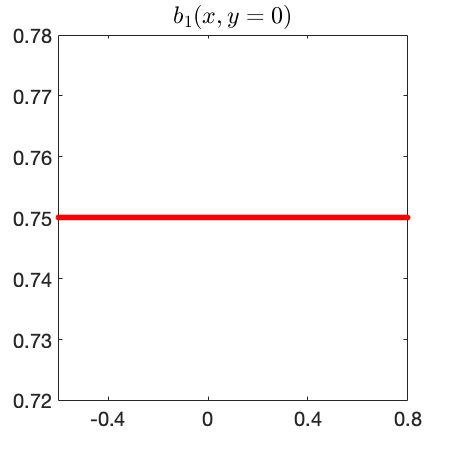}\hspace*{0.5cm}
            \includegraphics[trim=0.4cm 0.8cm 1.0cm 0.2cm, clip, width=5.2cm]{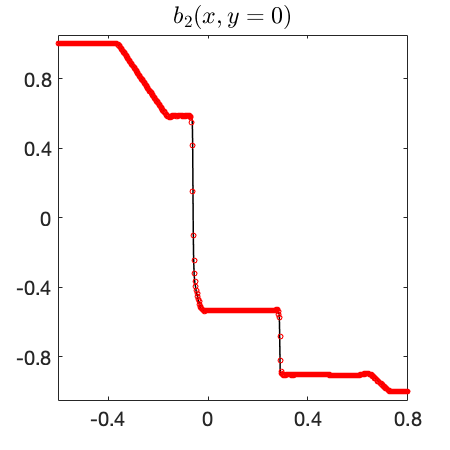}}
\caption{\sf Example 1: $\rho$, $b_1$, and $b_2$ computed by the PCCU scheme on $800\times8$ (red circles) and $1600\times16$ (solid black
line) uniform meshes.\label{fig1}}
\end{figure}

\paragraph{Example 2---Orszag-Tang Vortex Problem.} In the second example, we consider the Orszag-Tang vortex problem, which was introduced
in \cite{Orszag1979small} and has been widely used as a benchmark due to the formation and interaction of multiple shocks as the system
evolves in time and the presence of many important features of MHD turbulence; see, e.g.,
\cite{dumbser2019divergence,li2011central,li2012arbitrary,Liu2021New,Yakovlev2013locally}. The initial conditions for this problem read as
$$
\begin{aligned}
&\rho(x,y,0)\equiv\gamma^2,\quad u(x,y,0)=-\sin y,\quad v(x,y,0)=\sin x,\quad w(x,y,0)\equiv0,\\
&b_1(x,y,0)=-\sin y,\quad b_2(x,y,0)=\sin(2x),\quad b_3(x,y,0)\equiv0,\quad p(x,y,0)\equiv\gamma,
\end{aligned}
$$
where $\gamma=5/3$ is the specific heat ratio. We set the periodic boundary conditions on all sides of the computational domain
$[0,2\pi]\times[0,2\pi]$. 

The time evolution of the fluid density $\rho$ computed on a uniform $200\times200$ mesh is shown at times $t=0.5$, 2, 3, and 4 in Figure 
\ref{fig2}. We observe that the numerical solution computed by the proposed PCCU scheme remains stable and is consistent with previous 
results presented in \cite{li2011central,li2012arbitrary,Liu2021New,Yakovlev2013locally}, demonstrating that the ability of our scheme to
capture both smooth flows and shocks.
\begin{figure}[ht!]
\centerline{\includegraphics[trim=1.4cm 1.1cm 1.4cm 0.5cm, clip, width=5.3cm]{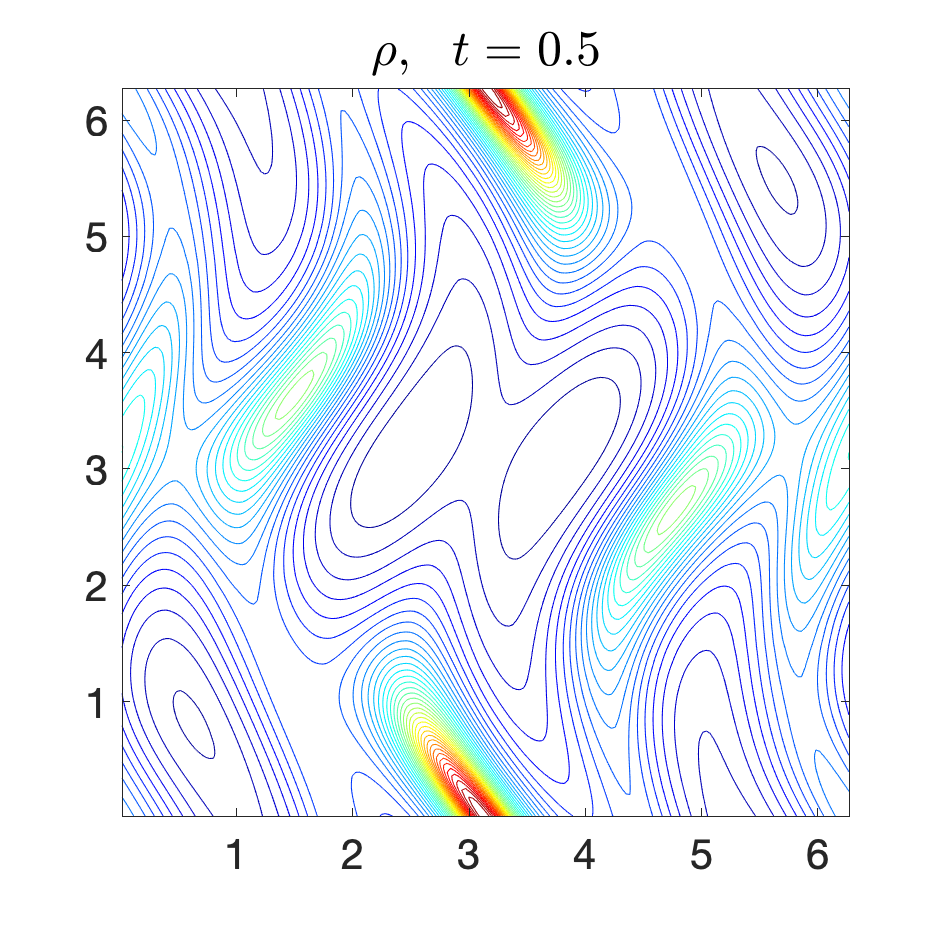}\hspace*{0.7cm}
            \includegraphics[trim=1.4cm 1.1cm 1.4cm 0.5cm, clip, width=5.3cm]{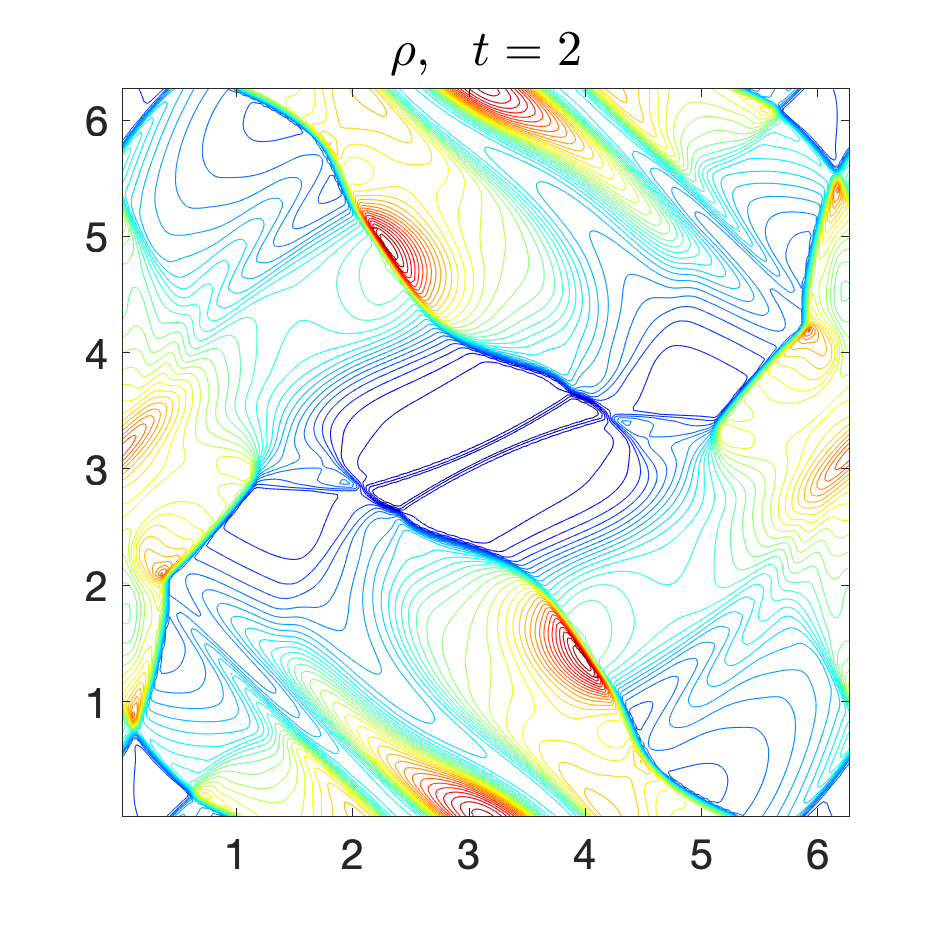}}
\vskip8pt
\centerline{\includegraphics[trim=1.4cm 1.1cm 1.4cm 0.5cm, clip, width=5.3cm]{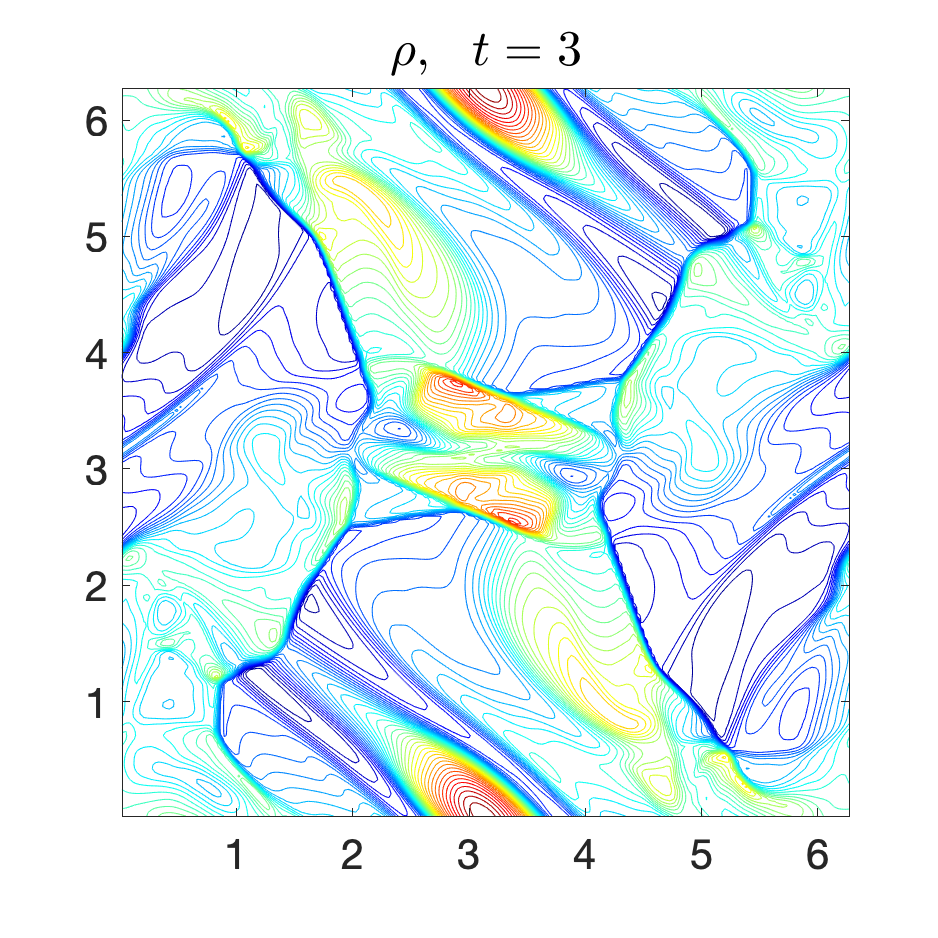}\hspace*{0.7cm}
            \includegraphics[trim=1.4cm 1.1cm 1.4cm 0.5cm, clip, width=5.3cm]{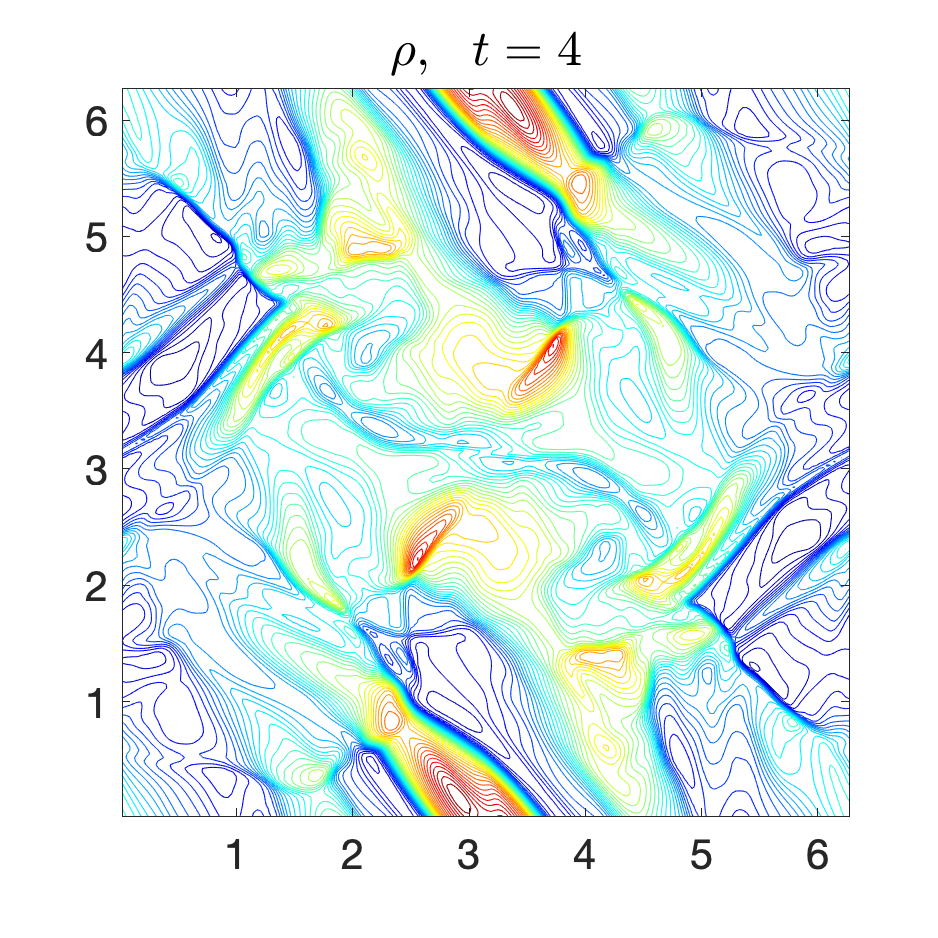}}
\caption{\sf Example 2: Fluid density $\rho$ computed by the proposed PCCU scheme at different times. 40 equally spaced contours are used in
each plot with the ranges $[2.11,5.83]$, $[0.63,6.17]$, $[1.29,6.12]$, and $[1.25,5.8]$, respectively.\label{fig2}}
\end{figure}

\paragraph{Example 3---Rotor Problem.} Next, we consider the ``second rotor problem'' from \cite{balsara1999staggered,toth2000constraint},
referred to as the rotor problem in this paper. This commonly used benchmark problem describes a rapidly-rotating disk of dense fluid
centered in a background of static fluid. Over time, the disk expands and rotates. The initial conditions are given by 
$$
\begin{aligned}
&(\rho,u,v)(x,y,0)=\left\{\begin{aligned}
&\left(10,\frac{0.5-y}{r_0},\frac{x-0.5}{r_0}\right),&&r<0.1,\\
&\left(1+9\lambda,\frac{\lambda(0.5-y)}{r},\frac{\lambda(x-0.5)}{r}\right),&&0.1\le r\le0.115,\\
&(1,0,0),&&r>0.115,
\end{aligned}\right.\\
&w(x,y,0)=b_2(x,y,0)=b_3(x,y,0)\equiv0,\quad b_1(x,y,0)\equiv\frac{2.5}{\sqrt{4\pi}},\quad p(x,y,0)\equiv0.5,
\end{aligned}
$$
where $r=\sqrt{(x-0.5)^2+(y-0.5)^2}$ and $\lambda=(0.115-r)/0.015$. We take the specific heat ratio $\gamma=5/3$ and use the periodic
boundary conditions on all sides of the computational domain $[0,1]\times[0,1]$.

In Figure \ref{fig3}, we show the fluid density $\rho$, pressure $p$, Mach number $\abs{\bm u}/c_s$ (where $c_s=\sqrt{\gamma p/\rho}$ is the
speed of sound), and magnetic pressure $\abs{\bm b}^2/2$ computed on a uniform $200\times200$ mesh at time $t=0.295$. We note that our
results are in good agreement with those reported in, e.g., \cite{toth2000constraint,Liu2021New,li2011central}. In addition, it is
emphasized in \cite{balsara1999staggered,Liu2021New,toth2000constraint} that, due to rapid changes at the center of the rotation, many
numerical methods produce oscillations or negative pressure values. We stress that during numerical simulations, we have not observed any 
oscillations, and the proposed PCCU scheme has produced no negative values of the computed pressure. The oscillation-free feature is further
illustrated in Figure \ref{fig3b}, where we zoom in on the center of the Mach number plots produced on three consecutively refined grids. 
\begin{figure}[ht!]
\centerline{\includegraphics[trim=0.4cm 1.1cm 1.4cm 0.4cm, clip, width=5.8cm]{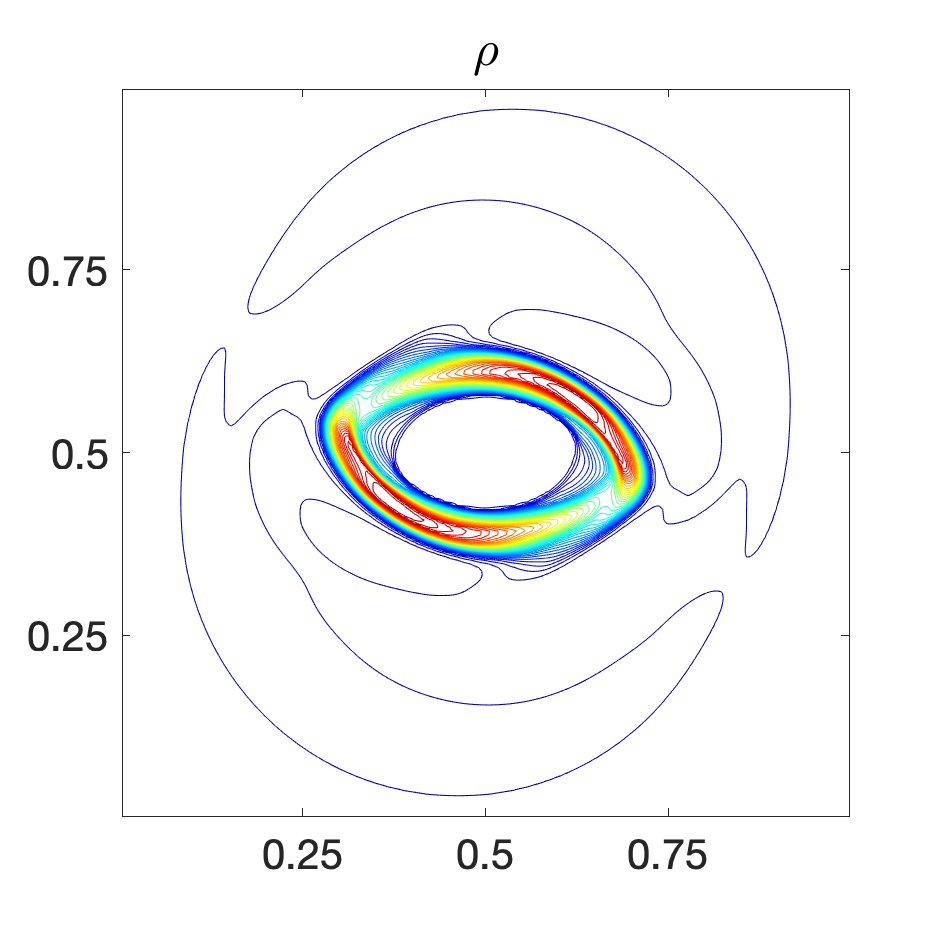}\hspace*{0.5cm}
            \includegraphics[trim=0.4cm 1.1cm 1.4cm 0.4cm, clip, width=5.8cm]{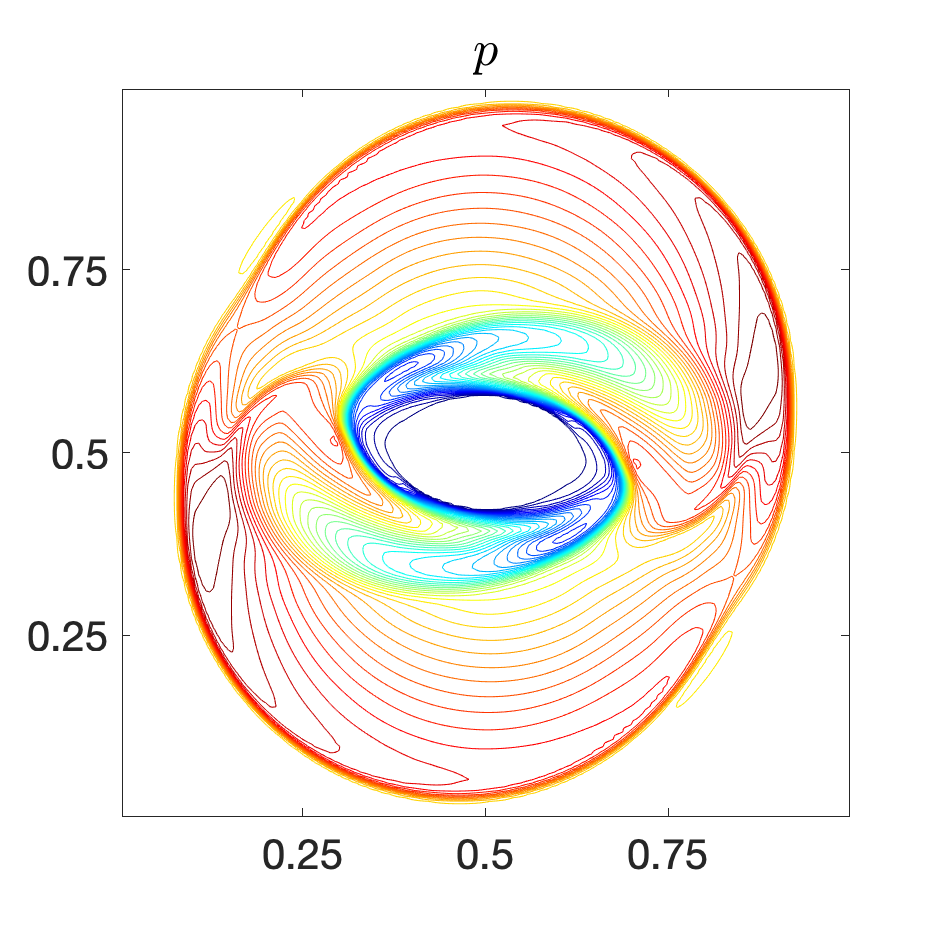}}
\vskip8pt
\centerline{\includegraphics[trim=0.4cm 1.1cm 1.4cm 0.4cm, clip, width=5.8cm]{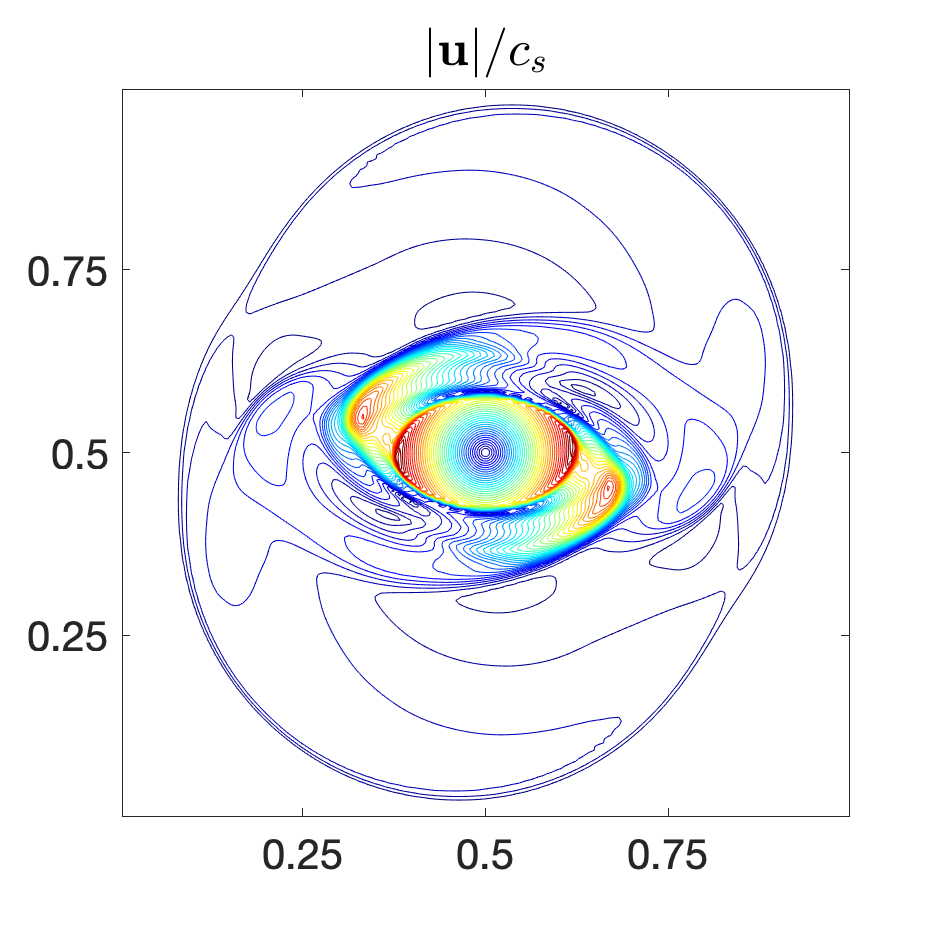}\hspace*{0.5cm}
            \includegraphics[trim=0.4cm 1.1cm 1.4cm 0.4cm, clip, width=5.8cm]{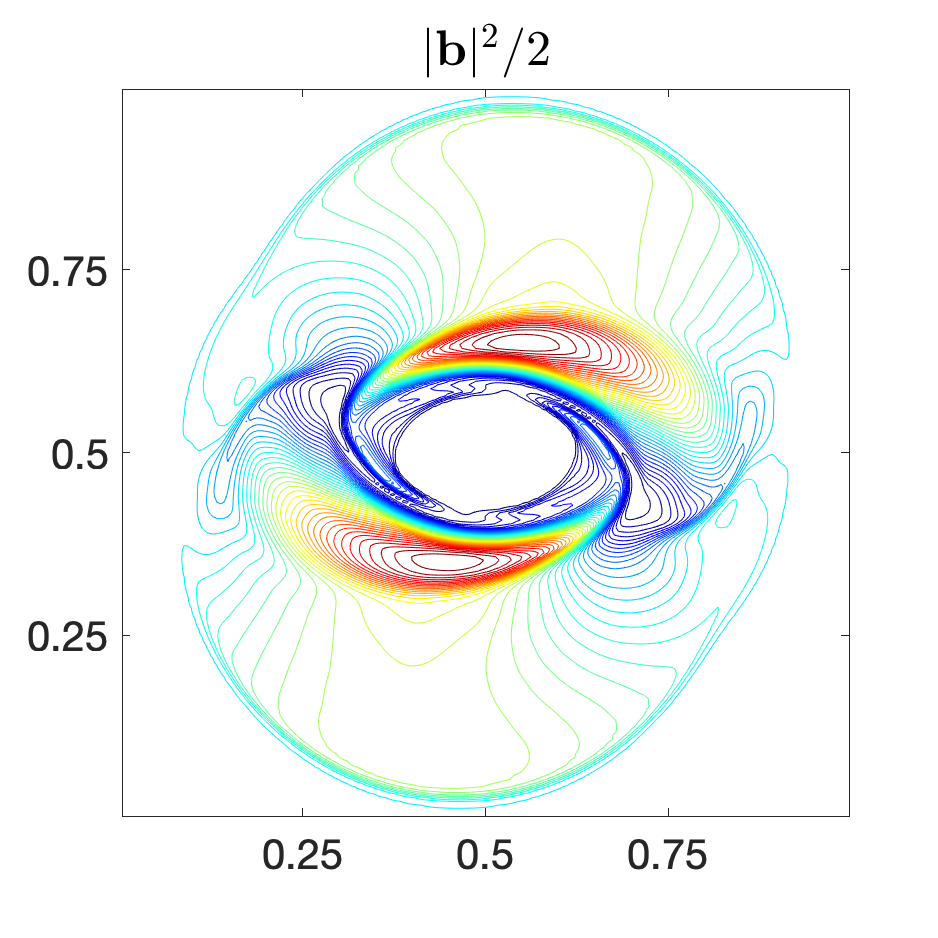}}
\caption{\sf Example 3: Fluid density $\rho$, pressure $p$, Mach number $\abs{\bm u}/c_s$, and magnetic pressure $\abs{\bm b}^2/2$ computed 
by the proposed PCCU scheme. 40 equally spaced contours are used in each plot with the ranges $[0.71,8.95]$, $[0.01,0.78]$, $[0,2.9]$, and
$[0.02,0.65]$, respectively.\label{fig3}}
\end{figure}
\begin{figure}[ht!]
\centerline{\includegraphics[trim=0.4cm 1.1cm 1.4cm 0.5cm, clip, width=5.3cm]{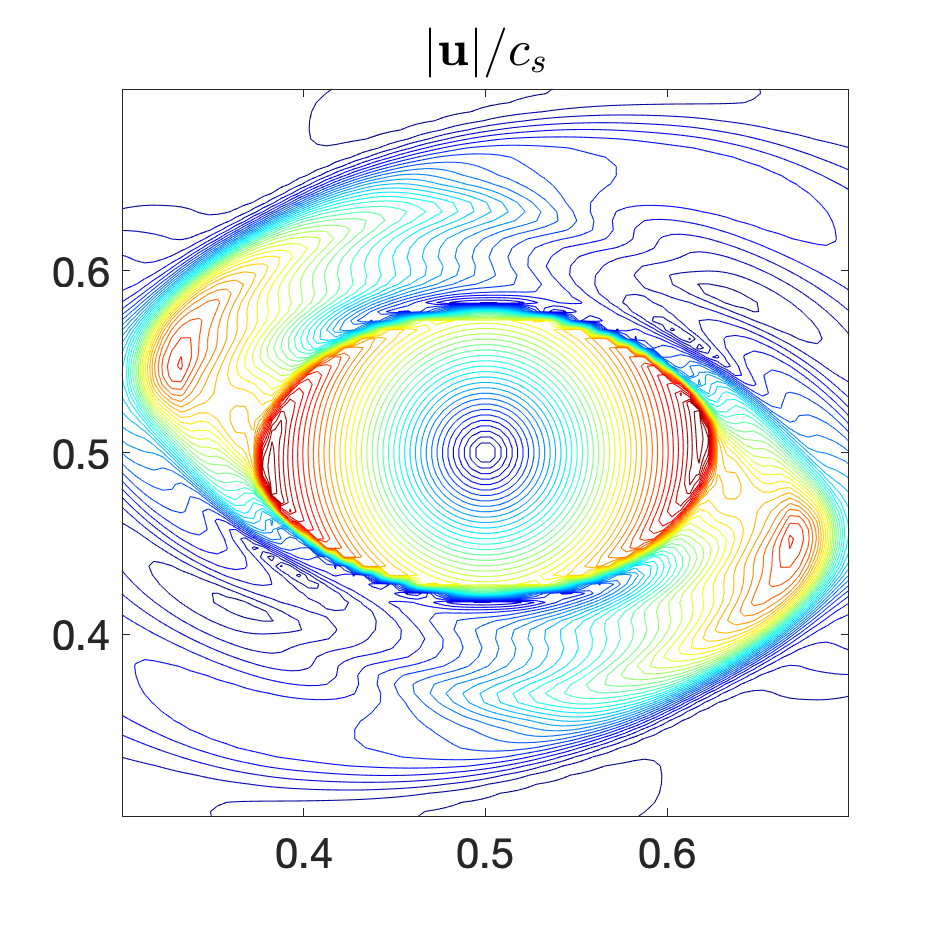}\hspace*{0.3cm}
            \includegraphics[trim=0.4cm 1.1cm 1.4cm 0.5cm, clip, width=5.3cm]{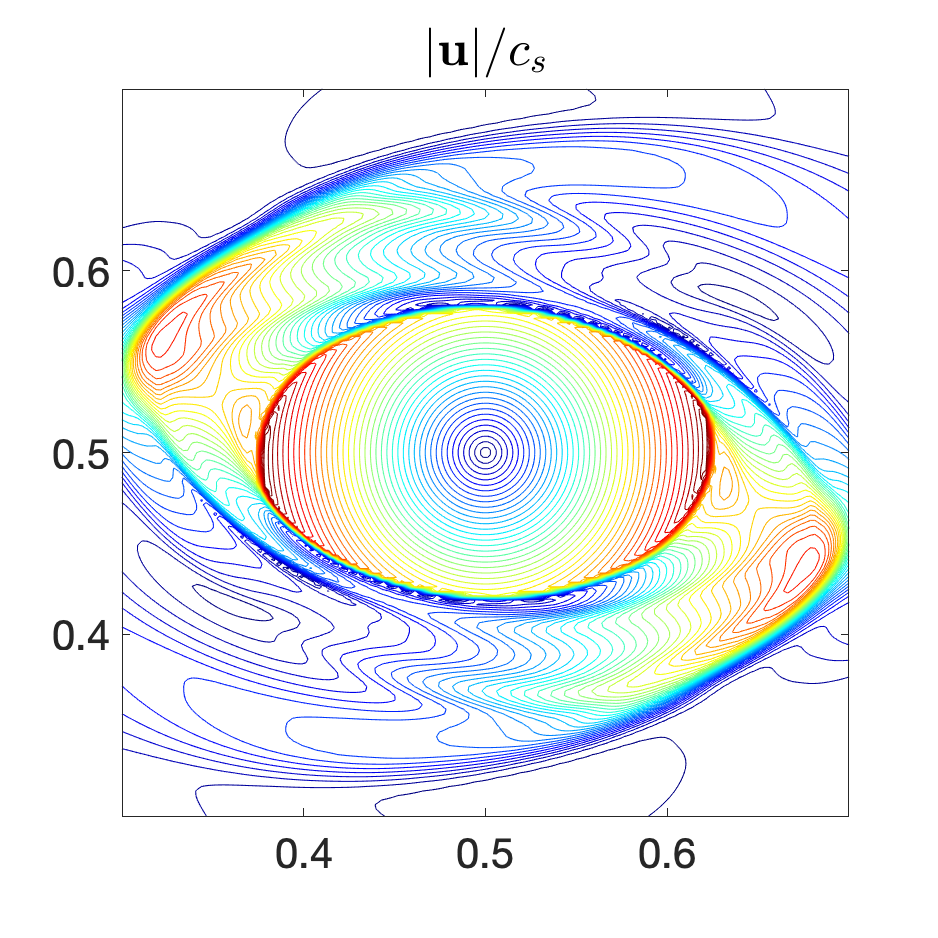}\hspace*{0.3cm}
            \includegraphics[trim=0.4cm 1.1cm 1.4cm 0.5cm, clip, width=5.3cm]{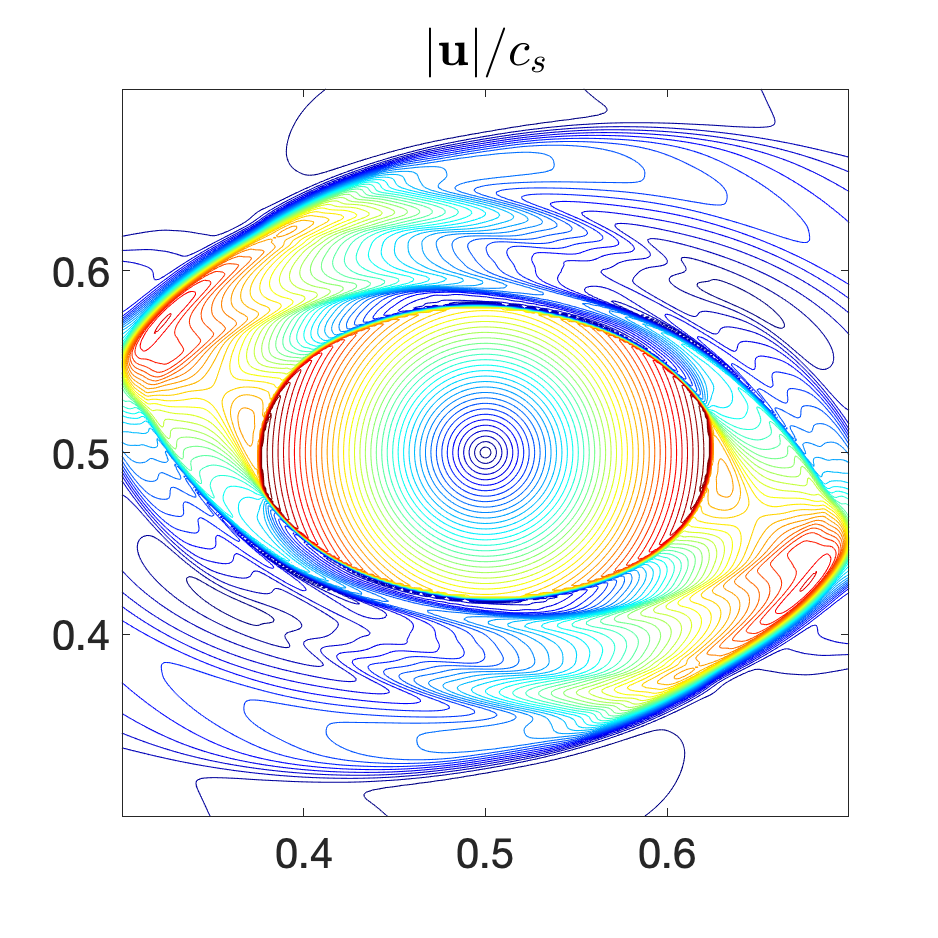}}
\caption{\sf Example 3: Zoom in on the center of the Mach number plots produced on uniform $200\times200$ (left), $400\times400$ (middle),
and $800\times800$ (right) meshes. 40 equally spaced contours are used in each plot.\label{fig3b}}
\end{figure}

\paragraph{Example 4---Blast Problem.} In this example, we consider the blast problem first introduced in \cite{balsara1999staggered}. This
benchmark problem is seen in a number of studies (see, e.g., \cite{li2011central,li2012arbitrary,Liu2021New,Yakovlev2013locally}) and is
considered a challenge due to the low gas pressure and strong magnetosonic shocks. Negative pressures are easily produced near the shocks;
see \cite{li2012arbitrary,li2011central} and references therein. The initial conditions are
$$
(\rho,u,v,w,b_1,b_2,b_3)(x,y,0)=\left(1,0,0,0,\frac{100}{\sqrt{4\pi}},0,0\right),\quad 
p(x,y,0)=\begin{cases}
1000,&\sqrt{x^2+y^2}<0.1,\\
0.1,&\mbox{othrwise}.
\end{cases}
$$
We take the specific heat ratio $\gamma=1.4$ and use zero-order extrapolation on the boundaries of the computational domain 
$[-0.5,0.5]\times[-0.5,0.5]$. 

The fluid density $\rho$, pressure $p$, magnitude of velocity $\abs{\bm u}$, and magnetic pressure $\abs{\bm b}^2/2$ computed by the
proposed PCCU scheme on a $200\times200$ uniform mesh at $t=0.01$ are depicted in Figure \ref{fig5}. Additionally, the numerical
experimentation of the proposed method resulted in positive pressure values throughout the entire computational domain, returning a minimum
pressure of 0.10. Positive pressure values are also completely maintained when running the blast problem on a refined $400\times400$ uniform
grid (the fine mesh results are not shown here for brevity).
\begin{figure}[ht!]
\centerline{\includegraphics[trim=0.2cm 1.1cm 1.4cm 0.4cm, clip, width=5.8cm]{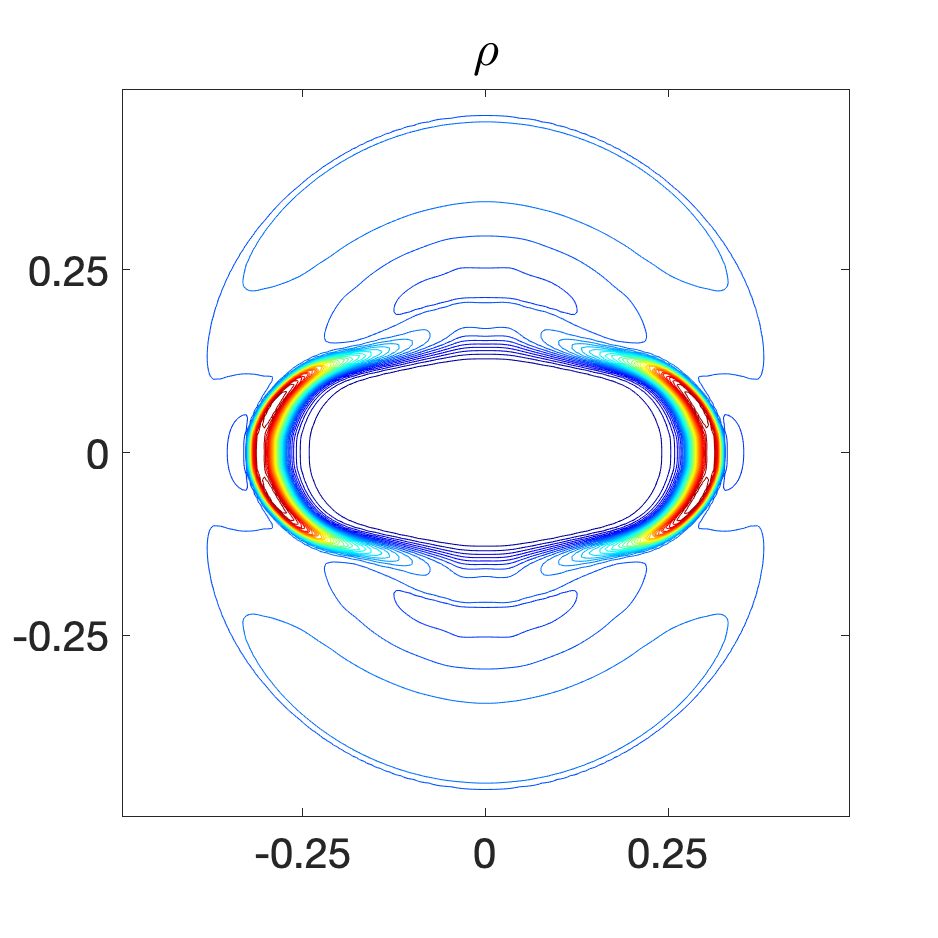}\hspace*{0.5cm}
            \includegraphics[trim=0.2cm 1.1cm 1.4cm 0.4cm, clip, width=5.8cm]{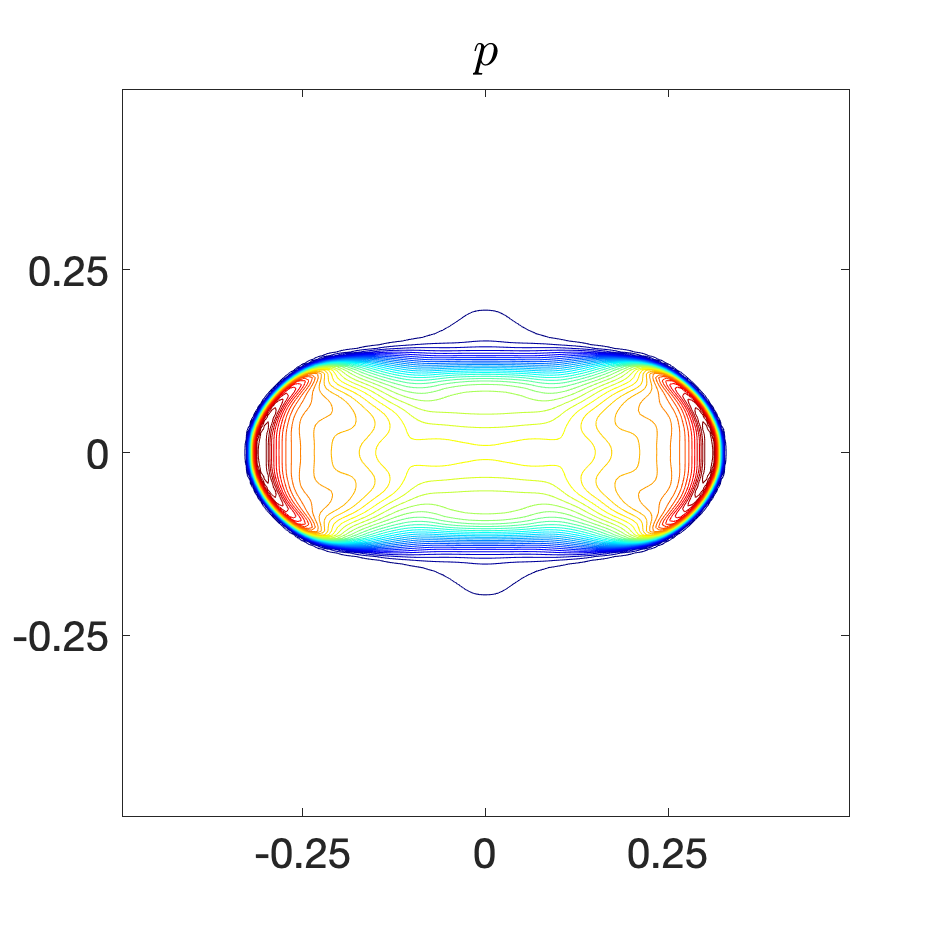}}
\vskip8pt
\centerline{\includegraphics[trim=0.2cm 1.1cm 1.4cm 0.4cm, clip, width=5.8cm]{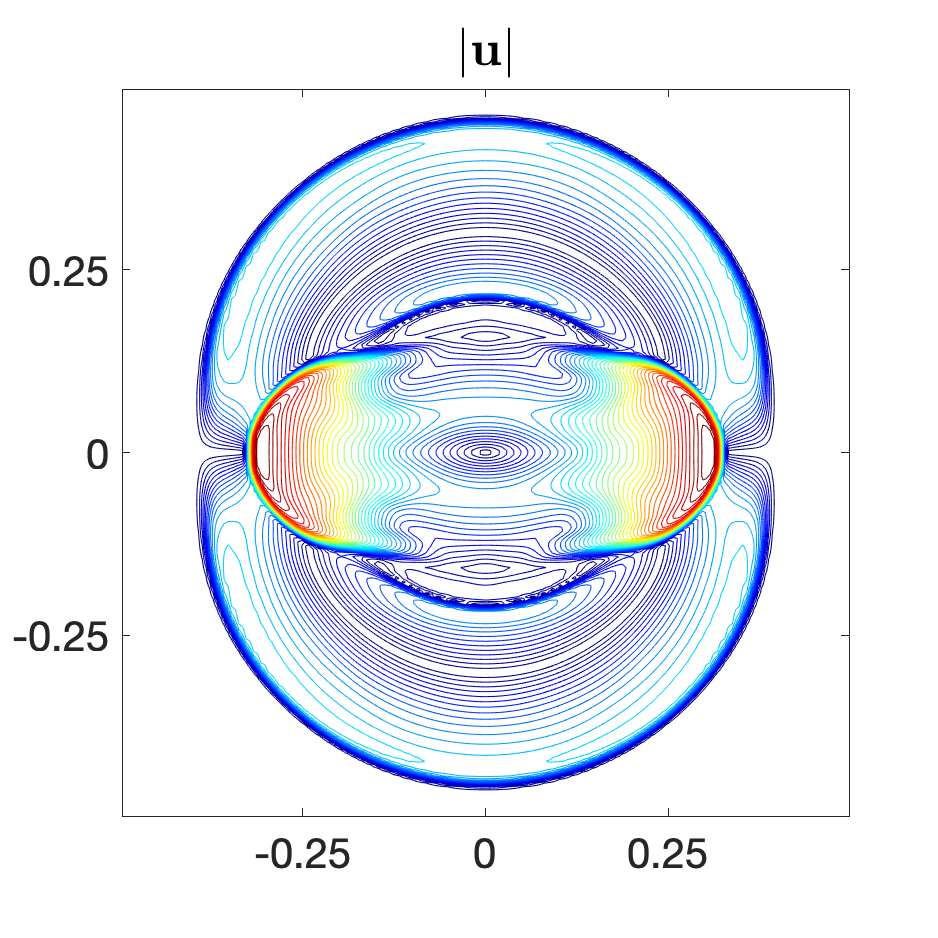}\hspace*{0.5cm}
            \includegraphics[trim=0.2cm 1.1cm 1.4cm 0.4cm, clip, width=5.8cm]{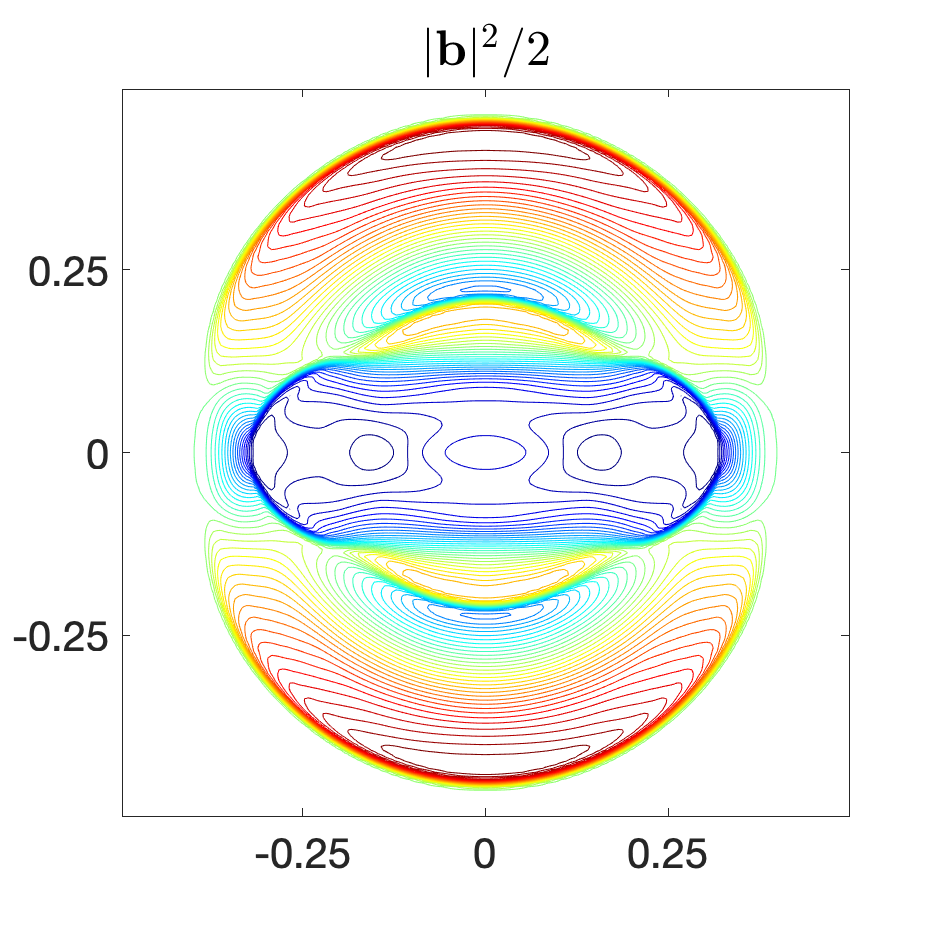}}
\caption{\sf Example 4: Fluid density $\rho$, pressure $p$, velocity magnitude $\abs{\bm u}$, and magnetic pressure $\abs{\bm b}^2/2$
computed by the proposed PCCU scheme. 40 equally spaced contours are used in each plot with the ranges $[0.22,4.09]$, $[0.10,250]$,
$[0,16.77]$, and $[215,588]$, respectively.\label{fig5}}
\end{figure}

\section{Shallow Water MHD}\label{sec3}
\subsection{Governing Equations}\label{sec31}
In this section, we study the 2-D modified Godunov-Powell shallow water MHD system, which reads as
\begin{equation}
\begin{aligned}
&h_t+(hu)_x+(hv)_y=0,\\
&(hu)_t+\left(hu^2+\frac{g}{2}h^2-ha^2\right)_x+(huv-hab)_y=-a\left[(ha)_x+(hb)_y\right],\\
&(hv)_t+(huv-hab)_x+\left(hv^2+\frac{g}{2}h^2-hb^2\right)_y=-b\left[(ha)_x+(hb)_y\right],\\
&(ha)_t+(hbu-hav)_y=-u\left[(ha)_x+(hb)_y\right],\\
&(hb)_t+(hav-hbu)_x=-v\left[(ha)_x+(hb)_y\right].
\end{aligned}
\label{3.1}
\end{equation}
Here, $h$ is the fluid thickness, $u$ and $v$ represent the $x$- and $y$-velocity, $(a,b)^\top$ is the reduced magnetic field, which has
units of velocity, and $g$ is the acceleration due to gravity. As in the ideal MHD system considered in \S\ref{sec2}, one can easily show
that
\begin{equation}
(ha)_x+(hb)_y=0
\label{3.2}
\end{equation}
as long as the field $(ha,hb)^\top$ is initially divergence-free. Therefore, the Godunov-Powell source terms on the RHS of \eref{3.1} are
theoretically zero. Still, they are added to the original shallow water MHD system (whose RHS is identically zero in the case of a flat bottom
topography) to help enforce the divergence-free constraint \eref{3.2} numerically; see, e.g.,
\cite{fuchs2011approximate,janhunen2000positive,waagan2011robust}.

In order to develop a locally divergence-free numerical method for the system \eref{3.1}, this divergence constraint \eref{3.2} must be
enforced on the discrete level. As in \S\ref{sec2}, we achieve this goal by introducing the new variables $A:=(ha)_x$ and $B:= (hb)_y$,
differentiating the induction equation in \eref{3.1} with respect to $x$ and $y$, and obtaining the following equations for $A$ and $B$,
\begin{equation}
\begin{aligned}
&A_t+\big(uA-hbu_y\big)_x+\big(vA+hav_x\big)_y=0,\\
&B_t+\big(uB+hbu_y\big)_x+\big(vB-hav_x\big)_y=0.
\end{aligned}
\label{3.3}
\end{equation}
which are then added to the studied system \eref{3.1}.

Prior to introducing the numerical method for the augmented system \eref{3.1}, \eref{3.3}, we write it in the vector form
\begin{equation}
\bm U_t+\bm F(\bm U)_x+\bm G(\bm U)_y= Q(\bm U)\bm U_x+R(\bm U)\bm U_y,
\label{3.8}
\end{equation}
where $\bm U:=(h,hu,hv,ha,hb,A,B)^\top$,
\allowdisplaybreaks
\begin{align}
&\bm F(\bm U)=
\begin{pmatrix}hu\\hu^2+\frac{g}{2}gh^2-ha^2\\huv-hab\\0\\hbu-hav\\uA-hbu_y\\uB+hbu_y\end{pmatrix},\quad
Q(\bm U)=\begin{pmatrix}0&0&0&0&0&0&0\\
0&0&0&-a&0&0&0\\
0&0&0&-b&0&0&0\\
0&0&0&-u&0&0&0\\
0&0&0&-v&0&0&0\\
0&0&0&0&0&0&0\\
0&0&0&0&0&0&0
\end{pmatrix},\nonumber\\
&\bm G(\bm U)=
\begin{pmatrix}hv\\huv-hab\\hv^2+\frac{g}{2}h^2-hb^2\\hav-hbu\\0\\vA+hav_x\\vB-hav_x,\end{pmatrix},\quad
R(\bm U)=\begin{pmatrix}0&0&0&0&0&0&0\\
0&0&0&0&-a&0&0\\
0&0&0&0&-b&0&0\\
0&0&0&0&-u&0&0\\
0&0&0&0&-v&0&0\\
0&0&0&0&0&0&0\\
0&0&0&0&0&0&0
\end{pmatrix}.\nonumber
\end{align}

\subsection{Numerical Method}\label{sec32}
We now extend the PCCU scheme developed in \S\ref{sec22} to the shallow water MHD system. 

Following the notation from section \ref{sec22}, the semi-discrete PCCU scheme still reads as \eref{2.13}--\eref{2.18}, and the resulting
system of ODEs is to be numerically integrated using an appropriate ODE solver, for instance, the three-stage third-order SSP Runge-Kutta,
which we have used in the numerical experiments reported in \S\ref{sec33}.

In \S\ref{sec321}, \S\ref{sec322}, and \S\ref{sec323} below, we focus on details of the scheme, which are different from the ideal MHD case.

\subsubsection{Piecewise Linear Reconstruction}\label{sec321}
A piecewise linear reconstruction is performed for the discrete variables
$$
\bm W_{j,k}:=(\,\xbar h_{j,k},u_{j,k},v_{j,k},(\xbar{ha})_{j,k},(\xbar{hb})_{j,k},\xbar A_{j,k},\xbar B_{j,k})^\top,
$$
where $u_{j,k}:=\nicefrac{(\,\xbar{hu})_{j,k}}{\,\xbar h_{j,k}}$ and $v_{j,k}:=\nicefrac{(\,\xbar{hv})_{j,k}}{\,\xbar h_{j,k}}$. We then
calculate cell interface values $\bm W^{\rm E,\rm W,\rm N,\rm S}_{j,k}$ using \eref{2.15}. The slopes $(W^{(i)}_x)_{j,k}$ for $i\ne4$ and
$(W^{(i)}_y)_{j,k}$ for $i\ne5$ are computed using the generalized minmod limiter as in \eref{2.16}--\eref{2.17f}.

Like in the ideal MHD case, the slopes $((ha)_x)_{j,k}$ and $((hb)_y)_{j,k}$ are computed in a way that allows one to enforce the local
discrete divergence-free condition $((ha)_x)_{j,k}+((hb)_y)_{j,k}\equiv0$ for all $j,k$. To this end, we proceed as in \S\ref{sec221} and
set
\begin{equation*}
((ha)_x)_{j,k}=\sigma_{j,k}\,\xbar A_{j,k},\quad((hb)_y)_{j,k}=\sigma_{j,k}\,\xbar B_{j,k},
\end{equation*}
where 
$$
\sigma_{j,k}=\min\Big\{1,\sigma^x_{j,k},\sigma^y_{j,k}\Big\},
$$
$$
\begin{aligned}
&\sigma^x_{j,k}:=\left\{\begin{aligned}
&\min\bigg\{1,\frac{((\widehat{ha})_x)_{j,k}}{\xbar A_{j,k}}\bigg\}&&\mbox{if}~((\widehat{ha})_x)_{j,k}\,\xbar A_{j,k}>0,\\
&0&&\mbox{otherwise},
\end{aligned}\right.\\
&\sigma^y_{j,k}:=\left\{\begin{aligned}
&\min\bigg\{1,\frac{((\widehat{hb})_y)_{j,k}}{\xbar B_{j,k}}\bigg\}&&\mbox{if}~((\widehat{hb})_y)_{j,k}\,\xbar B_{j,k}>0,\\
&0&&\mbox{otherwise},
\end{aligned}\right.
\end{aligned}
$$
and $((\widehat{ha})_x)_{j,k}$ and $((\widehat{hb})_y)_{j,k}$ are computed using the generalized minmod reconstruction as in
\eref{2.16}--\eref{2.17f}.

\subsubsection{One-Sided Speeds of Propagation}\label{sec322}
We point out that in the shallow water MHD case, computing the one-sided speeds $s^\pm_{\jph,k}$ and $s^\pm_{j,\kph}$ needed in the
semi-discretization \eref{2.13}--\eref{2.18}, is significantly easier than in the ideal MHD case. We follow the general recipe and
estimate the $x$- and $y$-directional speeds using the largest and smallest eigenvalues of the matrices
$\frac{\partial\bm F}{\partial\bm U}(\bm U)-Q(\bm U)$ and $\frac{\partial\bm G}{\partial\bm U}(\bm U)-R(\bm U)$, respectively. This results
in
$$
\begin{aligned}
&s_{\jph,k}^+=\max\left\{u_{j,k}^{\rm E}+\sqrt{\big(a_{j,k}^{\rm E}\big)^2+gh_{j,k}^{\rm E}},\,
u_{j+1,k}^{\rm W}+\sqrt{\big(a_{j+1,k}^{\rm W}\big)^2+gh_{j+1,k}^{\rm W}},\,0\right\},\\
&s_{\jph,k}^-=\min\left\{u_{j,k}^{\rm E}-\sqrt{\big(a_{j,k}^{\rm E}\big)^2+gh_{j,k}^{\rm E}},\,
u_{j+1,k}^{\rm W}-\sqrt{\big(a_{j+1,k}^{\rm W}\big)^2+gh_{j+1,k}^{\rm W}},\,0\right\},\\
&s_{j,\kph}^+=\max\left\{v_{j,k}^{\rm N}+\sqrt{\big(b_{j,k}^{\rm N}\big)^2+gh_{j,k}^{\rm N}},\,
v_{j,k+1}^{\rm S}+\sqrt{\big(b_{j,k+1}^{\rm S}\big)^2+gh_{j,k+1}^{\rm S}},\,0\right\},\\
&s_{j,\kph}^-=\min\left\{v_{j,k}^{\rm N}-\sqrt{\big(b_{j,k}^{\rm N}\big)^2+gh_{j,k}^{\rm N}},\,
v_{j,k+1}^{\rm S}-\sqrt{\big(b_{j,k+1}^{\rm S}\big)^2+gh_{j,k+1}^{\rm S}},\,0\right\}.
\end{aligned}
$$

\subsubsection{Discretization of Nonconservative Products}\label{sec323}
In order to evaluate the contribution of the nonconservative terms $Q(\bm U)\bm U_x$ appearing on the RHS of \eref{3.8}, we again follow the
lines of \cite{CKN22,CKM} and evaluate the corresponding integrals exactly:
\allowdisplaybreaks
\begin{align*}
&Q^{(1)}_{j,k}=Q^{(6)}_{j,k}=Q^{(7)}_{j,k}=Q^{(1)}_{\bm\Psi,\jph,k}=Q^{(6)}_{\bm\Psi,\jph,k}=Q^{(7)}_{\bm\Psi,\jph,k}=0,\\
&Q^{(2)}_{j,k}=-\int\limits_{x_\jmh}^{x_\jph}\frac{\widetilde{ha}(x,y_k)}{\widetilde h(x,y_k)}((ha)_x)_{j,k}\,{\rm d}x\\
&\hspace*{0.7cm}=\left\{\begin{aligned}
&-a_{j,k}\sigma_{j,k}\,\xbar A_{j,k}\dx&&\mbox{if}~(h_x)_{j,k}=0,\\
&-\sigma_{j,k}\,\xbar A_{j,k}\left(\frac{(\xbar{ha})_{j,k}(h_x)_{j,k}-\,\xbar h_{j,k}\sigma_{j,k}\,\xbar A_{j,k}}{((h_x)_{j,k})^2}\,
\ln\bigg(\frac{h^{\rm E}_{j,k}}{h^{\rm W}_{j,k}}\bigg)+\frac{\sigma_{j,k}\,\xbar A_{j,k}\dx}{(h_x)_{j,k}}\right)&&\mbox{otherwise},
\end{aligned}\right.\\
&Q^{(3)}_{j,k}=-\int\limits_{x_\jmh}^{x_\jph}\frac{\widetilde{hb}(x,y_k)}{\widetilde h(x,y_k)}((ha)_x)_{j,k}\,{\rm d}x\\
&\hspace*{0.7cm}=\left\{\begin{aligned}
&-b_{j,k}\sigma_{j,k}\,\xbar A_{j,k}\dx&&\mbox{if}~(h_x)_{j,k}=0,\\
&-\sigma_{j,k}\,\xbar A_{j,k}\left(\frac{(\xbar{hb})_{j,k}(h_x)_{j,k}-\,\xbar h_{j,k}((hb)_x)_{j,k}}{((h_x)_{j,k})^2}\,
\ln\bigg(\frac{h^{\rm E}_{j,k}}{h^{\rm W}_{j,k}}\bigg)+\frac{((hb)_x)_{j,k}\dx}{(h_x)_{j,k}}\right)&&\mbox{otherwise},
\end{aligned}\right.\\
&Q^{(4)}_{j,k}=-\int\limits_{x_\jmh}^{x_\jph}\widetilde u(x,y_k)((ha)_x)_{j,k}\,{\rm d}x=-u_{j,k}\sigma_{j,k}\,\xbar A_{j,k}\dx,\\
&Q^{(5)}_{j,k}=-\int\limits_{x_\jmh}^{x_\jph}\widetilde v(x,y_k)((ha)_x)_{j,k}\,{\rm d}x=-v_{j,k}\sigma_{j,k}\,\xbar A_{j,k}\dx,\\
&Q^{(2)}_{\bm\Psi,\jph,k}=-\int\limits_0^1\frac{(ha)^{\rm E}_{j,k}+s\p{(ha)^{\rm W}_{j+1,k}-(ha)^{\rm E}_{j,k}}}
{h^{\rm E}_{j,k}+s\p{h^{\rm W}_{j+1,k}-h^{\rm E}_{j,k}}}[ha]_{\jph,k}\,{\rm d}s\\
&\hspace*{0.7cm}=\left\{\begin{aligned}
&-\hf\left(a^{\rm E}_{j,k}+a^{\rm W}_{j+1,k}\right)[ha]_{\jph,k}&&\mbox{if}~[h]_{\jph,k}=0,\\
&-[ha]_{\jph,k}\left(\frac{(ha)^{\rm E}_{j,k}[h]_{\jph,k}-h^{\rm E}_{j,k}[ha]_{\jph,k}}{[h]_{\jph,k}^2}\,
\ln\bigg(\frac{h^{\rm W}_{j+1,k}}{h^{\rm E}_{j,k}}\bigg)+\frac{[ha]_{\jph,k}}{[h]_{\jph,k}}\right)&&\mbox{otherwise},
\end{aligned}\right.\\
&Q^{(3)}_{\bm\Psi,\jph,k}=-\int\limits_0^1\frac{(hb)^{\rm E}_{j,k}+s\p{(hb)^{\rm W}_{j+1,k}-(hb)^{\rm E}_{j,k}}}
{h^{\rm E}_{j,k}+s\p{h^{\rm W}_{j+1,k}-h^{\rm E}_{j,k}}}[ha]_{\jph,k}\,{\rm d}s\\
&\hspace*{0.7cm}=\left\{\begin{aligned}
&-\hf\left(b^{\rm E}_{j,k}+b^{\rm W}_{j+1,k}\right)[ha]_{\jph,k}&&\mbox{if}~[h]_{\jph,k}=0,\\
&-[ha]_{\jph,k}\left(\frac{(hb)^{\rm E}_{j,k}[h]_{\jph,k}-h^{\rm E}_{j,k}[hb]_{\jph,k}}{[h]_{\jph,k}^2}\,
\ln\bigg(\frac{h^{\rm W}_{j+1,k}}{h^{\rm E}_{j,k}}\bigg)+\frac{[hb]_{\jph,k}}{[h]_{\jph,k}}\right)&&\mbox{otherwise},
\end{aligned}\right.\\
&Q^{(4)}_{\bm\Psi,\jph,k}=-\int\limits_0^1\left\{u^{\rm E}_{j,k}+s\p{u^{\rm W}_{j+1,k}-u^{\rm E}_{j,k}}\right\}[ha]_{\jph,k}\,{\rm d}s=
-\hf\left(u^{\rm E}_{j,k}+u^{\rm W}_{j+1,k}\right)[ha]_{\jph,k},\\
&Q^{(5)}_{\bm\Psi,\jph,k}=-\int\limits_0^1\left\{v^{\rm E}_{j,k}+s\p{v^{\rm W}_{j+1,k}-v^{\rm E}_{j,k}}\right\}[ha]_{\jph,k}\,{\rm d}s=
-\hf\left(v^{\rm E}_{j,k}+v^{\rm W}_{j+1,k}\right)[ha]_{\jph,k},
\end{align*}
where
$$
\begin{aligned}
&a_{j,k}:=\frac{(\,\xbar{ha})_{j,k}}{\,\xbar h_{j,k}},\quad b_{j,k}:=\frac{(\,\xbar{hb})_{j,k}}{\,\xbar h_{j,k}},\quad
a^{\rm E(W)}_{j,k}:=\frac{(ha)^{\rm E(W)}_{j,k}}{h^{\rm E(W)}_{j,k}},\quad
b^{\rm E(W)}_{j,k}:=\frac{(hb)^{\rm E(W)}_{j,k}}{h^{\rm E(W)}_{j,k}},\\
&[h]_{\jph,k}:=h^{\rm W}_{j+1,k}-h^{\rm E}_{j,k},\quad[ha]_{\jph,k}:=(ha)^{\rm W}_{j+1,k}-(ha)^{\rm E}_{j,k},\quad
[hb]_{\jph,k}:=(hb)^{\rm W}_{j+1,k}-(hb)^{\rm E}_{j,k}.
\end{aligned}
$$

The contribution of the nonconservative terms $R(\bm U)\bm U_x$ appearing on the RHS of \eref{3.8} is obtained in a similar manner and given
by
\allowdisplaybreaks
\begin{align*}
&R^{(1)}_{j,k}=R^{(6)}_{j,k}=R^{(7)}_{j,k}=R^{(1)}_{\bm\Psi,j,\kph}=R^{(6)}_{\bm\Psi,j,\kph}=R^{(7)}_{\bm\Psi,j,\kph}=0,\\
&R^{(2)}_{j,k}=
\left\{\begin{aligned}
&-a_{j,k}\sigma_{j,k}\,\xbar B_{j,k}\dy&&\mbox{if}~(h_y)_{j,k}=0,\\
&-\sigma_{j,k}\,\xbar B_{j,k}\left(\frac{(\xbar{ha})_{j,k}(h_y)_{j,k}-\,\xbar h_{j,k}((ha)_y)_{j,k}}{((h_y)_{j,k})^2}\,
\ln\bigg(\frac{h^{\rm N}_{j,k}}{h^{\rm S}_{j,k}}\bigg)+\frac{((ha)_y)_{j,k}\dy}{(h_y)_{j,k}}\right)&&\mbox{otherwise},
\end{aligned}\right.\\[0.7ex]
&R^{(3)}_{j,k}=\left\{\begin{aligned}
&-b_{j,k}\sigma_{j,k}\,\xbar B_{j,k}\dy&&\mbox{if}~(h_y)_{j,k}=0,\\
&-\sigma_{j,k}\,\xbar B_{j,k}\left(\frac{(\xbar{hb})_{j,k}(h_y)_{j,k}-\,\xbar h_{j,k}\sigma_{j,k}\,\xbar B_{j,k}}{((h_y)_{j,k})^2}\,
\ln\bigg(\frac{h^{\rm N}_{j,k}}{h^{\rm S}_{j,k}}\bigg)+\frac{\sigma_{j,k}\,\xbar B_{j,k}\dy}{(h_y)_{j,k}}\right)&&\mbox{otherwise},
\end{aligned}\right.\\
&R^{(4)}_{j,k}=-u_{j,k}\sigma_{j,k}\,\xbar B_{j,k}\dy,\quad R^{(5)}_{j,k}=-v_{j,k}\sigma_{j,k}\,\xbar B_{j,k}\dy,\\
&R^{(2)}_{\bm\Psi,j,\kph}=
\left\{\begin{aligned}
&-\hf\left(a^{\rm N}_{j,k}+a^{\rm S}_{j,k+1}\right)[hb]_{j,\kph}&&\mbox{if}~[h]_{j,\kph}=0,\\
&-[hb]_{j,\kph}\left(\frac{(ha)^{\rm N}_{j,k}[h]_{j,\kph}-h^{\rm N}_{j,k}[ha]_{j,\kph}}{[h]_{j,\kph}^2}\,
\ln\bigg(\frac{h^{\rm S}_{j,k+1}}{h^{\rm N}_{j,k}}\bigg)+\frac{[ha]_{j,\kph}}{[h]_{j,\kph}}\right)&&\mbox{otherwise},
\end{aligned}\right.\\
&R^{(3)}_{\bm\Psi,j,\kph}=\left\{\begin{aligned}
&-\hf\left(b^{\rm N}_{j,k}+b^{\rm S}_{j,k+1}\right)[hb]_{j,\kph}&&\mbox{if}~[h]_{j,\kph}=0,\\
&-[hb]_{j,\kph}\left(\frac{(hb)^{\rm N}_{j,k}[h]_{j,\kph}-h^{\rm N}_{j,k}[hb]_{j,\kph}}{[h]_{j,\kph}^2}\,
\ln\bigg(\frac{h^{\rm S}_{j,k+1}}{h^{\rm N}_{j,k}}\bigg)+\frac{[hb]_{j,\kph}}{[h]_{j,\kph}}\right)&&\mbox{otherwise},
\end{aligned}\right.\\
&R^{(4)}_{\bm\Psi,j,\kph}=-\hf\left(u^{\rm N}_{j,k}+u^{\rm S}_{j,k+1}\right)[hb]_{j,\kph},\quad
R^{(5)}_{\bm\Psi,j,\kph}=-\hf\left(v^{\rm N}_{j,k}+v^{\rm S}_{j,k+1}\right)[hb]_{j,\kph}.
\end{align*}
where
$$
\begin{aligned}
&a^{\rm N(S)}_{j,k}:=\frac{(ha)^{\rm N(S)}_{j,k}}{h^{\rm N(S)}_{j,k}},\quad
b^{\rm N(S)}_{j,k}:=\frac{(hb)^{\rm N(S)}_{j,k}}{h^{\rm N(S)}_{j,k}},\quad[h]_{j,\kph}:=h^{\rm S}_{j,k+1}-h^{\rm N}_{j,k},\\
&[ha]_{j,\kph}:=(ha)^{\rm S}_{j,k+1}-(ha)^{\rm N}_{j,k},\quad[hb]_{j,\kph}:=(hb)^{\rm S}_{j,k+1}-(hb)^{\rm N}_{j,k}.
\end{aligned}
$$

\subsection{Numerical Examples}\label{sec33}
In this section, we apply the proposed PCCU scheme to the 2-D shallow water MHD equations. In all of the examples, the CFL number is set to
0.25 and the minmod parameter is set to $\theta=1.3$.

\paragraph{Example 5---Orszag-Tang-Like Problem.} This example taken from \cite{Duan2021High,Zia2014Numerical} is similar to that of the
ideal MHD Orszag-Tang problem studied in Example 2.

The shallow water MHD system is considered in the computational domain $[0,2\pi]\times[0,2\pi]$ subject to the periodic boundary conditions
in both $x$ and $y$-directions and the following smooth initial data:
\begin{equation*}
(h,u,v,a,b)(x,y,0)=\left(\frac{25}{9},-\sin y,\sin x,-\sin y,\sin(2x)\right).
\end{equation*}

We compute the numerical solution by the proposed PCCU scheme on a uniform $200\times200$ mesh until the final time $t=2$. Time snapshots of
$h$ and $\sqrt{a^2+b^2}$ at $t=1$ and 2 are plotted in Figure \ref{fig6}. As one can see, the initially smooth solution breaks down and
develops multiple shock waves, whose interaction leads to the appearance of many essential features of MHD turbulence. We observe that the
obtained results are in good agreement with those reported in \cite{Duan2021High,Zia2014Numerical}.
\begin{figure}[ht!]
\centerline{\includegraphics[trim=1.4cm 1.1cm 1.4cm 0.5cm, clip, width=5.3cm]{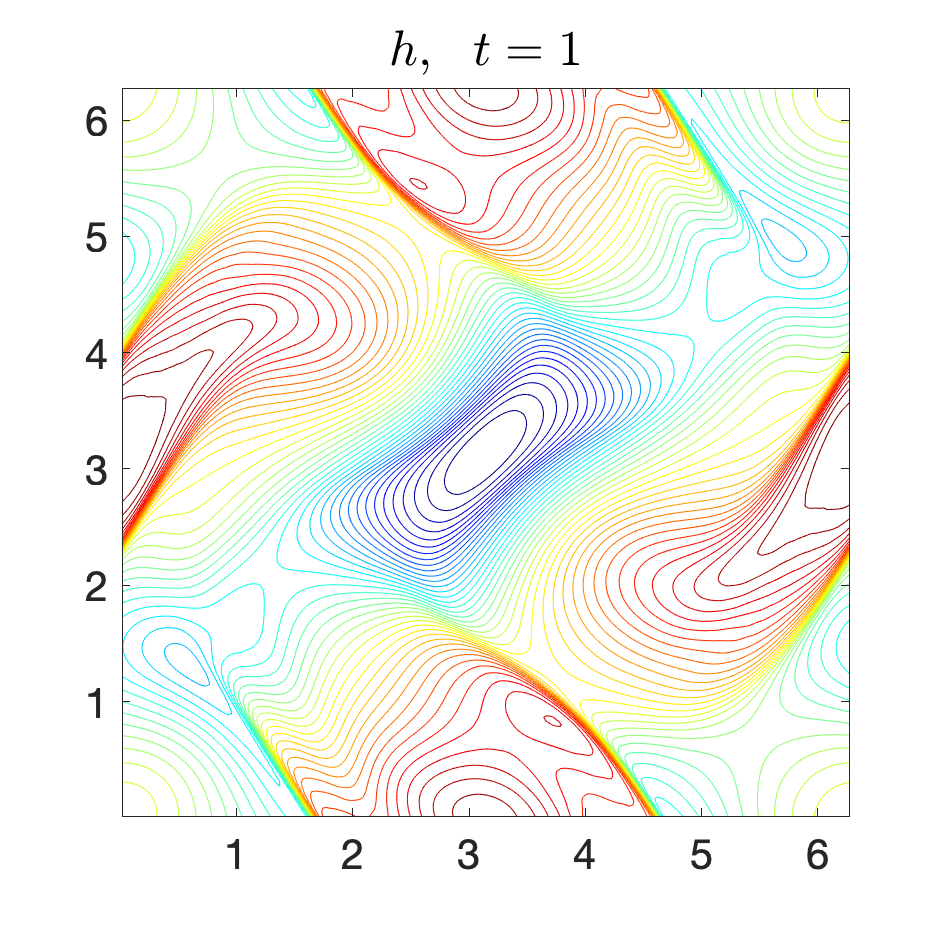}\hspace*{0.7cm}
            \includegraphics[trim=1.4cm 1.1cm 1.4cm 0.5cm, clip, width=5.3cm]{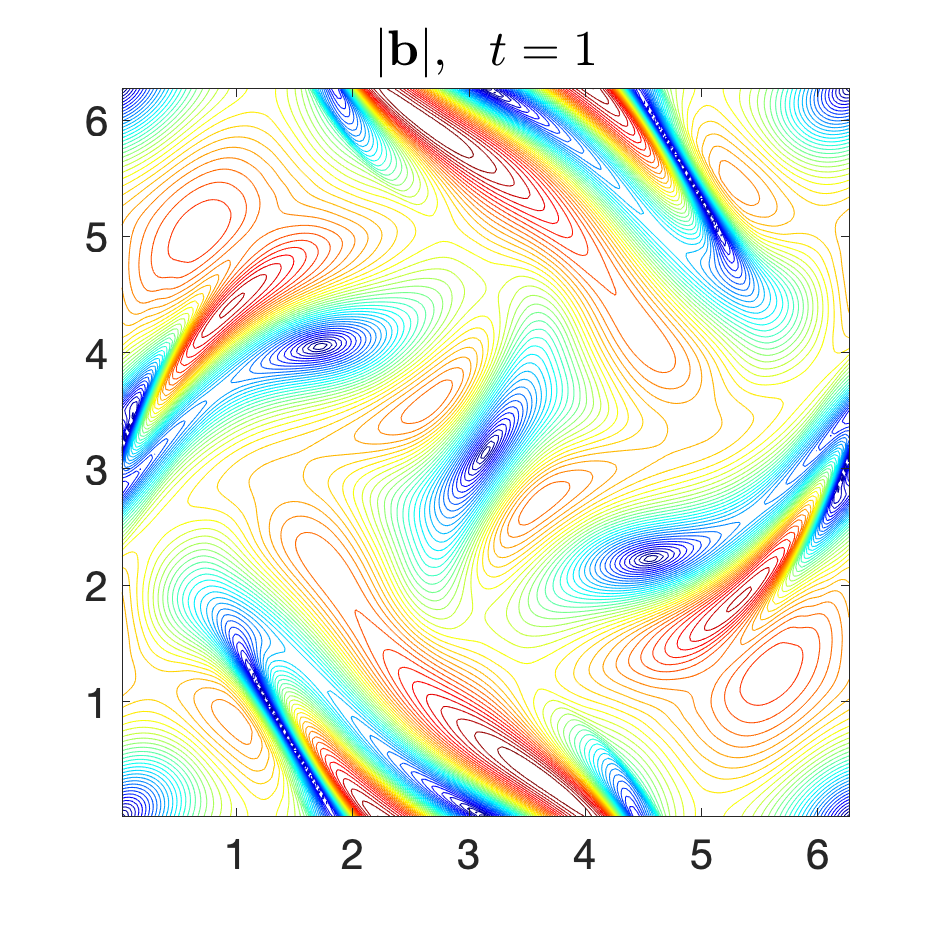}}
\vskip8pt
\centerline{\includegraphics[trim=1.4cm 1.1cm 1.4cm 0.5cm, clip, width=5.3cm]{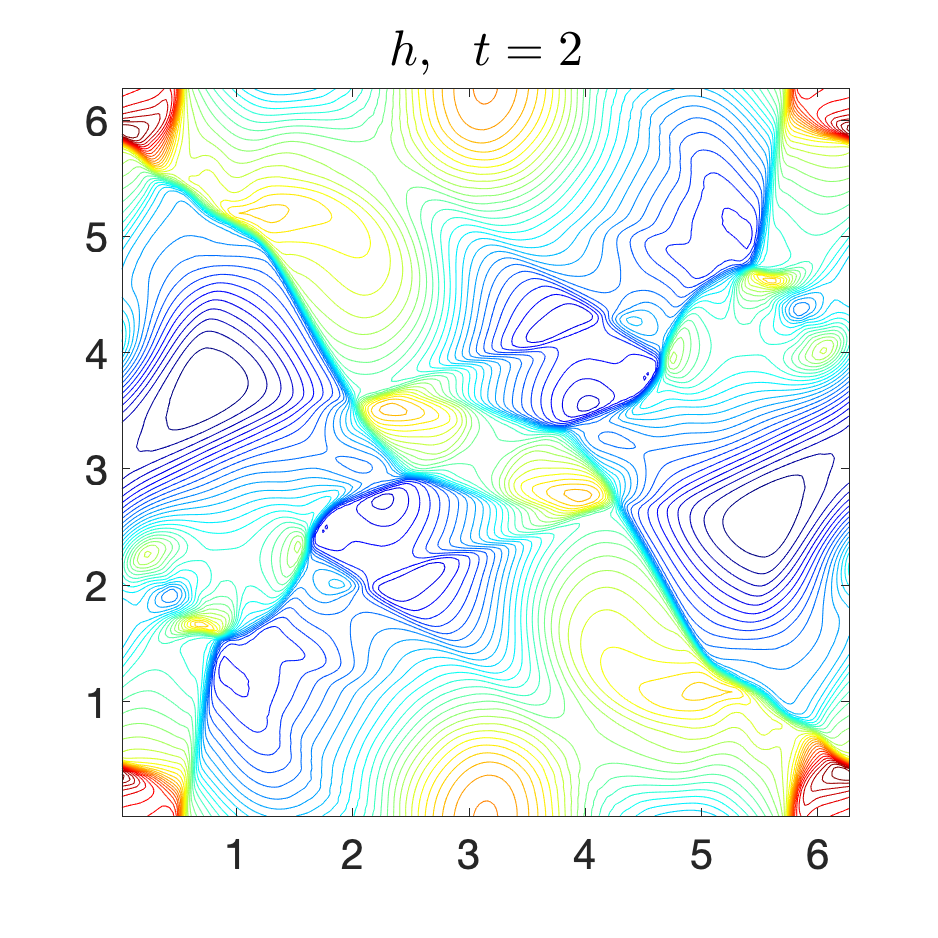}\hspace*{0.7cm}
            \includegraphics[trim=1.4cm 1.1cm 1.4cm 0.5cm, clip, width=5.3cm]{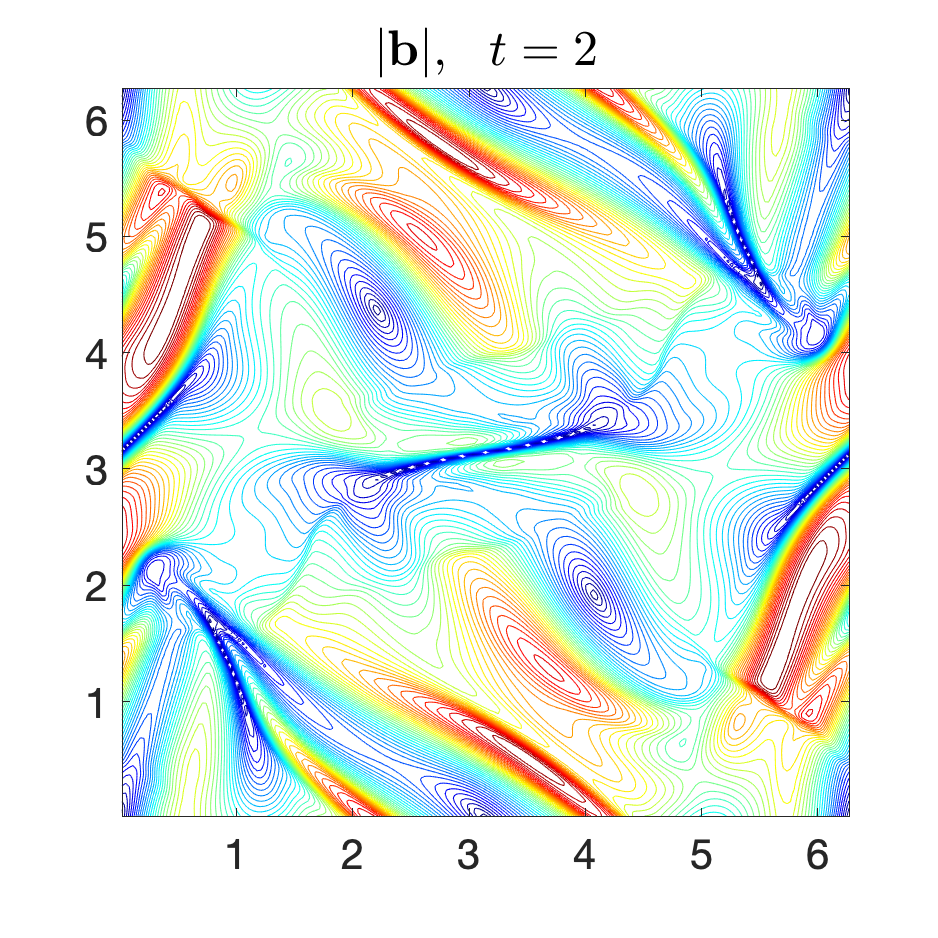}}
\caption{\sf Example 5: Fluid thickness $h$ and magnetic field magnitude $\sqrt{a^2+b^2}$ computed by the proposed PCCU scheme at $t=1$ (top
row) and 2 (bottom row). 40 equally spaced contours are used in each plot.\label{fig6}}
\end{figure}

\paragraph{Example 6---Rotor-Like Problem.} Next, we consider a rotor-like problem taken from \cite{Duan2021High,Kroger2005evolution}. This
benchmark, which is an extension of the ideal MHD rotor problem studied in Example 3, portrays a disk with radius 0.1 of significant fluid
depth $h$ rotating in a magnetic field. 

The initial data
\begin{equation*}
(h,u,v,ha,hb)=\begin{cases}(10,-y,x,1,0),&\sqrt{x^2+y^2}<0.1,\\(1,0,0,1,0),&\mbox{otherwise},\end{cases}
\end{equation*}
are prescribed in the computational domain $[-1,1]\times[-1,1]$ and zero-order extrapolation boundary conditions are set along its boundary.
The solution computed by the proposed PCCU scheme on a uniform $200\times200$ mesh at time $t=0.2$ is plotted Figure \ref{fig7}. The
obtained results are oscillation-free and overall comparable to those reported in \cite{Duan2021High,Kroger2005evolution}.
\begin{figure}[ht!]
\centerline{\includegraphics[trim=0.6cm 1.1cm 1.4cm 0.5cm, clip, width=5.3cm]{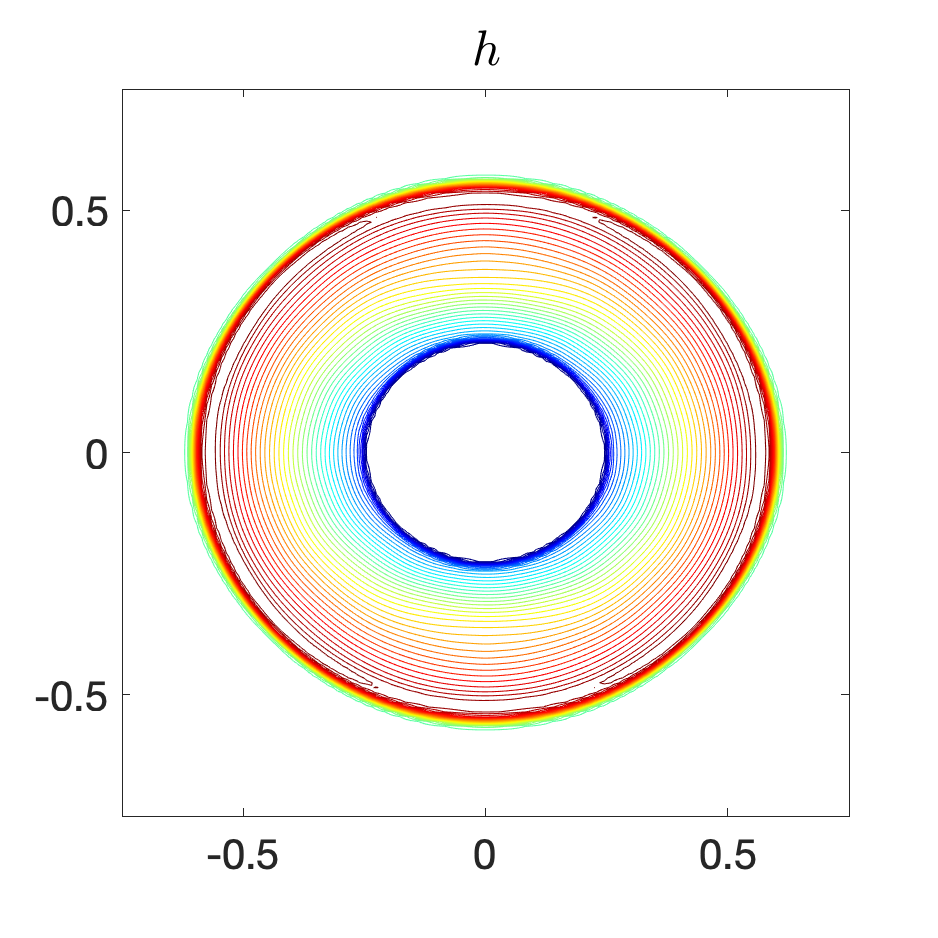}\hspace*{0.3cm}
            \includegraphics[trim=0.6cm 1.1cm 1.4cm 0.5cm, clip, width=5.3cm]{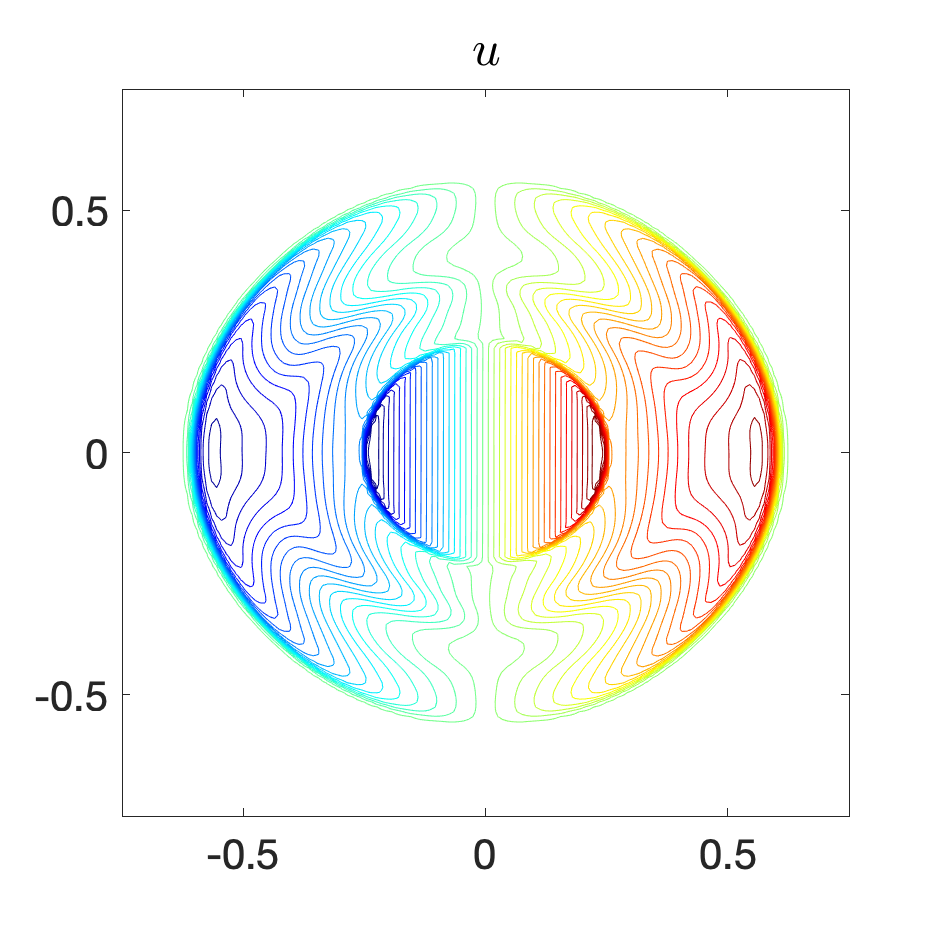}\hspace*{0.3cm}
            \includegraphics[trim=0.6cm 1.1cm 1.4cm 0.5cm, clip, width=5.3cm]{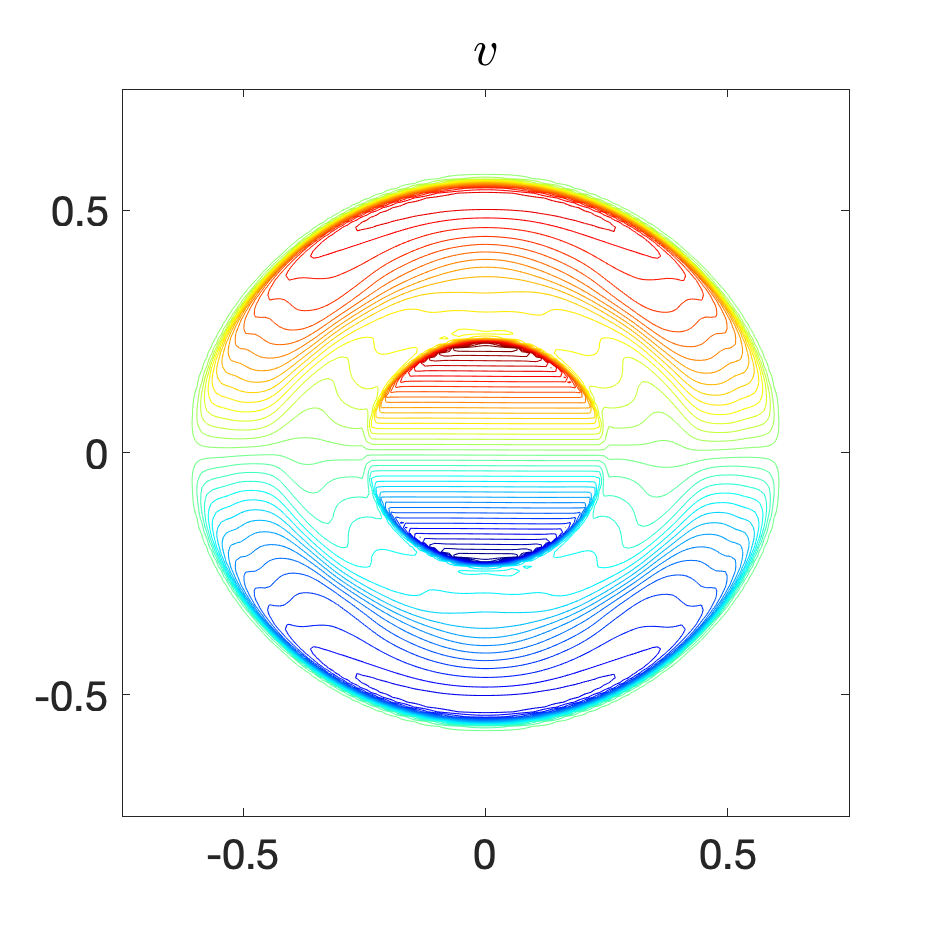}}
\vskip8pt
\centerline{\includegraphics[trim=0.6cm 1.1cm 1.4cm 0.5cm, clip, width=5.3cm]{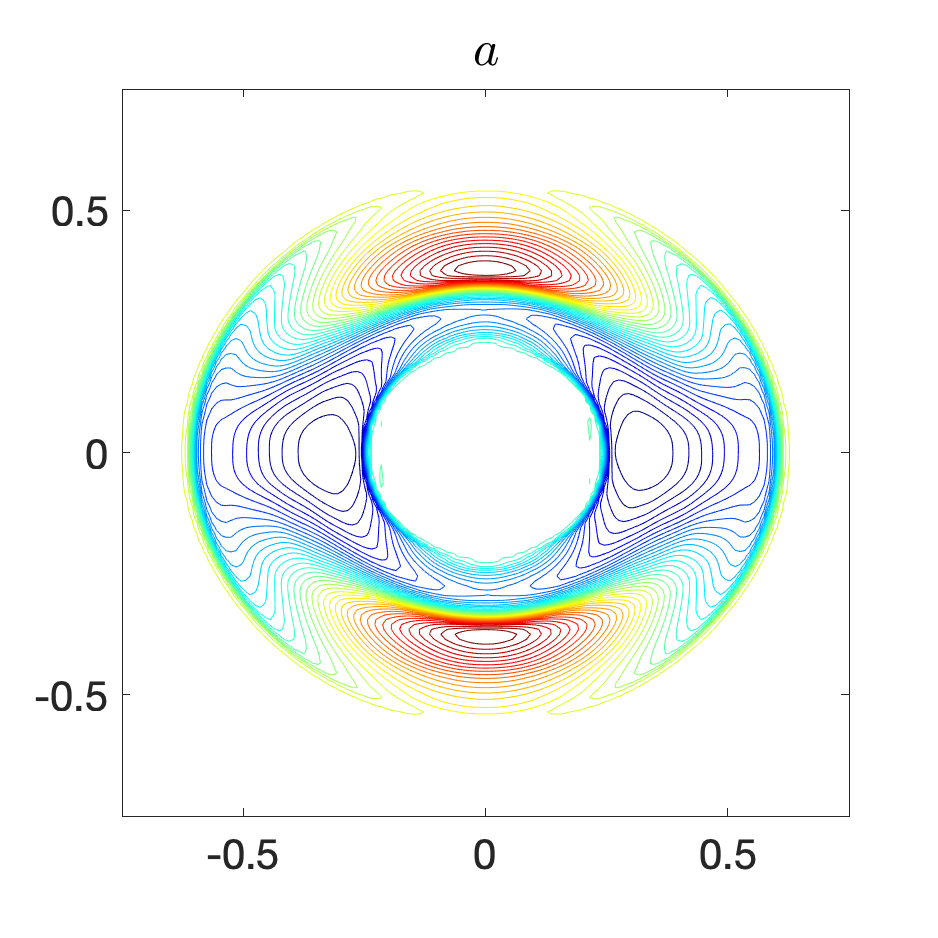}\hspace*{0.3cm}
            \includegraphics[trim=0.6cm 1.1cm 1.4cm 0.5cm, clip, width=5.3cm]{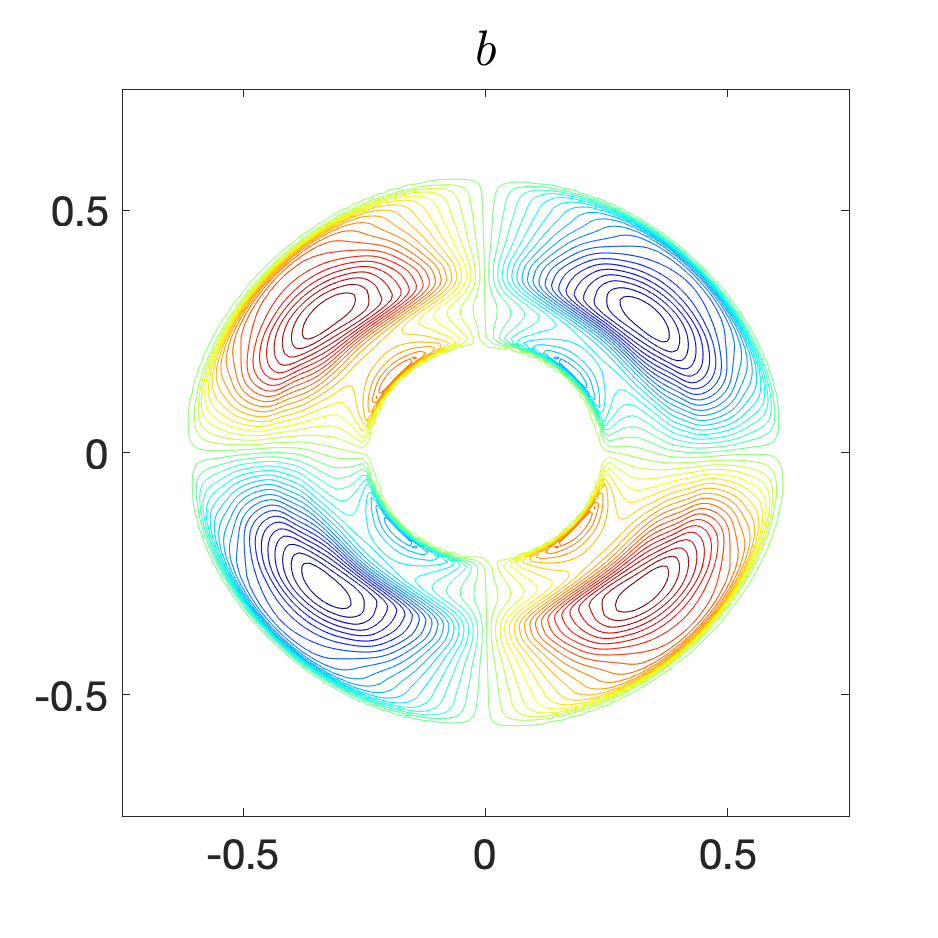}}
\caption{\sf Example 6: Fluid thickness $h$, velocities $u$ and $v$, reduced magnetic field components $a$ and $b$ computed by the proposed
PCCU scheme. 40 equally spaced contours are used in each plot.\label{fig7}}
\end{figure}

\paragraph{Example 7---Explosion Problem.} In the final example, we numerically solve the explosion problem studied in
\cite{Kroger2005evolution,Touma2010Unstaggered}. This is another benchmark for the shallow water MHD equations considered subject to the
following initial data: 
\begin{equation*}
(h,u,v,a,b)(x,y,0)=\begin{cases}(1,0,0,0.1,0),&\sqrt{x^2+y^2}<0.3,\\(0.1,0,0,1,0),&\mbox{otherwise}.\end{cases}
\end{equation*}
In this example, we take the computational domain $[-1,1]\times[-1,1]$ and implement the zero-order extrapolation boundary conditions along
its boundary.

The solution of the explosion problem consists of a shock traveling away from the center, a rarefaction wave traveling toward the origin,
and two Alfv\'en waves. We compute the solution by the proposed PCCU scheme on a uniform $200\times200$ mesh. The obtained results, shown in
Figure \ref{fig8} at $t=0.25$, are non-oscillatory and agree well with the corresponding results in
\cite{Kroger2005evolution,Touma2010Unstaggered}.
\begin{figure}[ht!]
\centerline{\includegraphics[trim=0.6cm 1.1cm 1.4cm 0.5cm, clip, width=5.3cm]{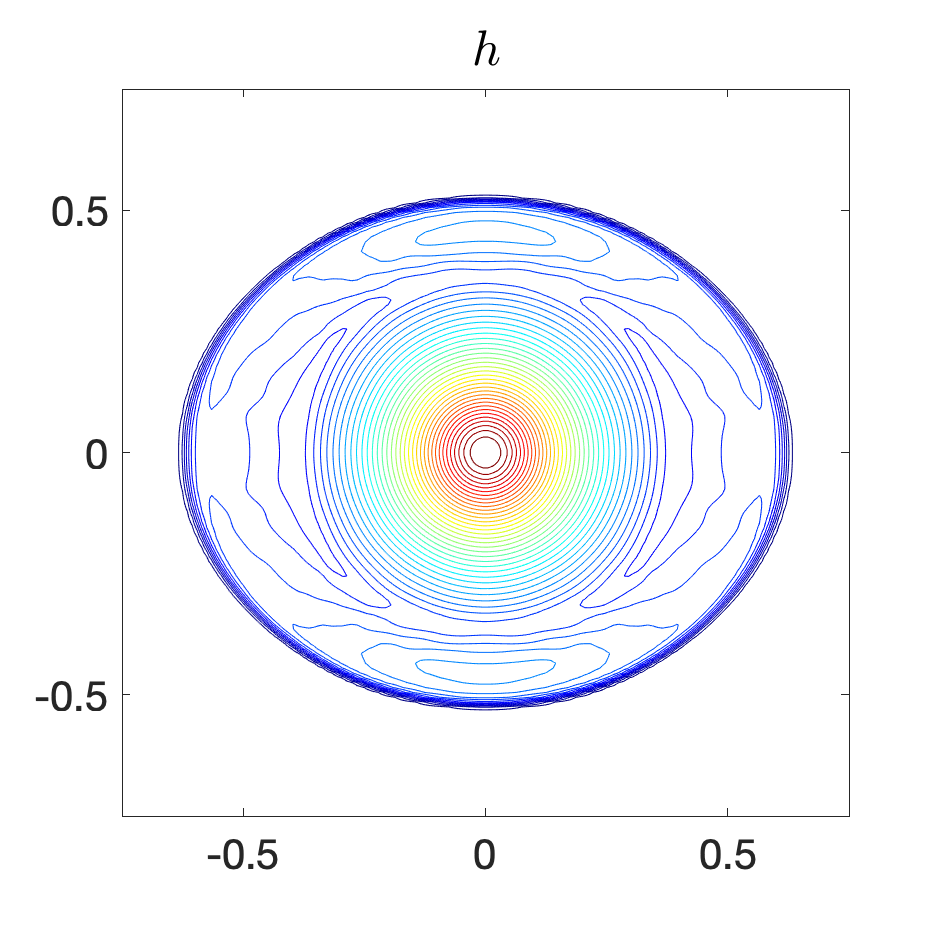}\hspace*{0.3cm}
            \includegraphics[trim=0.6cm 1.1cm 1.4cm 0.5cm, clip, width=5.3cm]{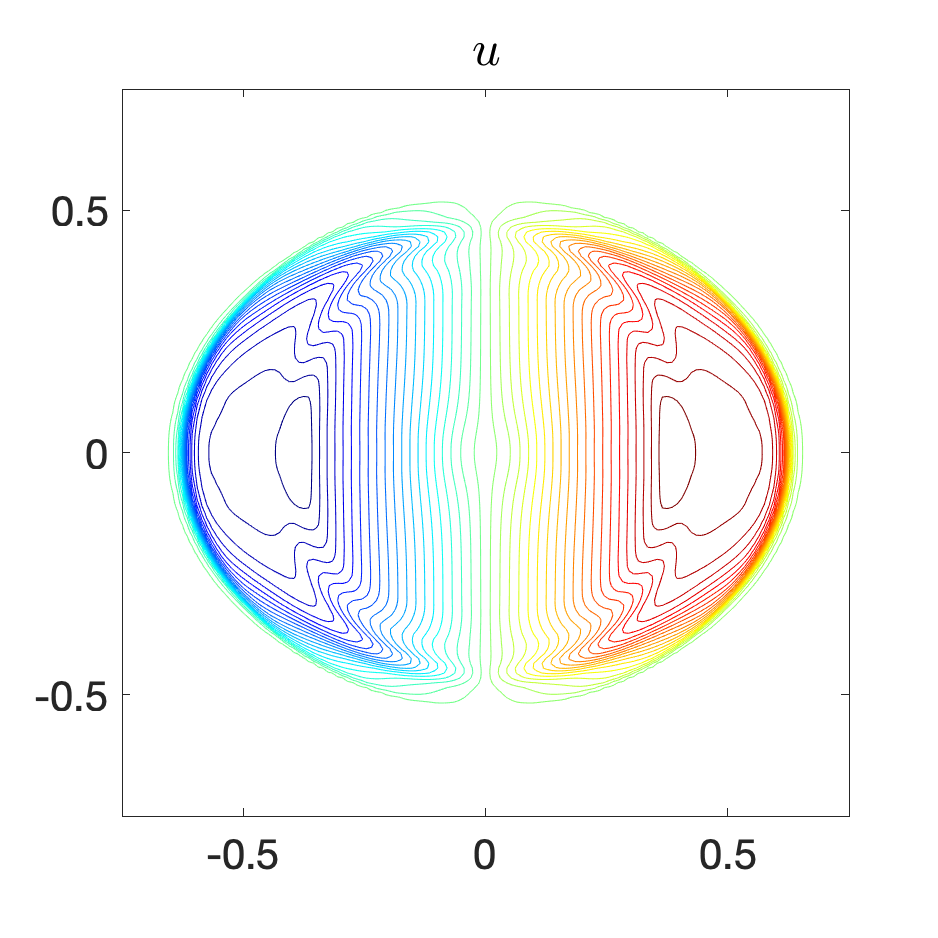}\hspace*{0.3cm}
            \includegraphics[trim=0.6cm 1.1cm 1.4cm 0.5cm, clip, width=5.3cm]{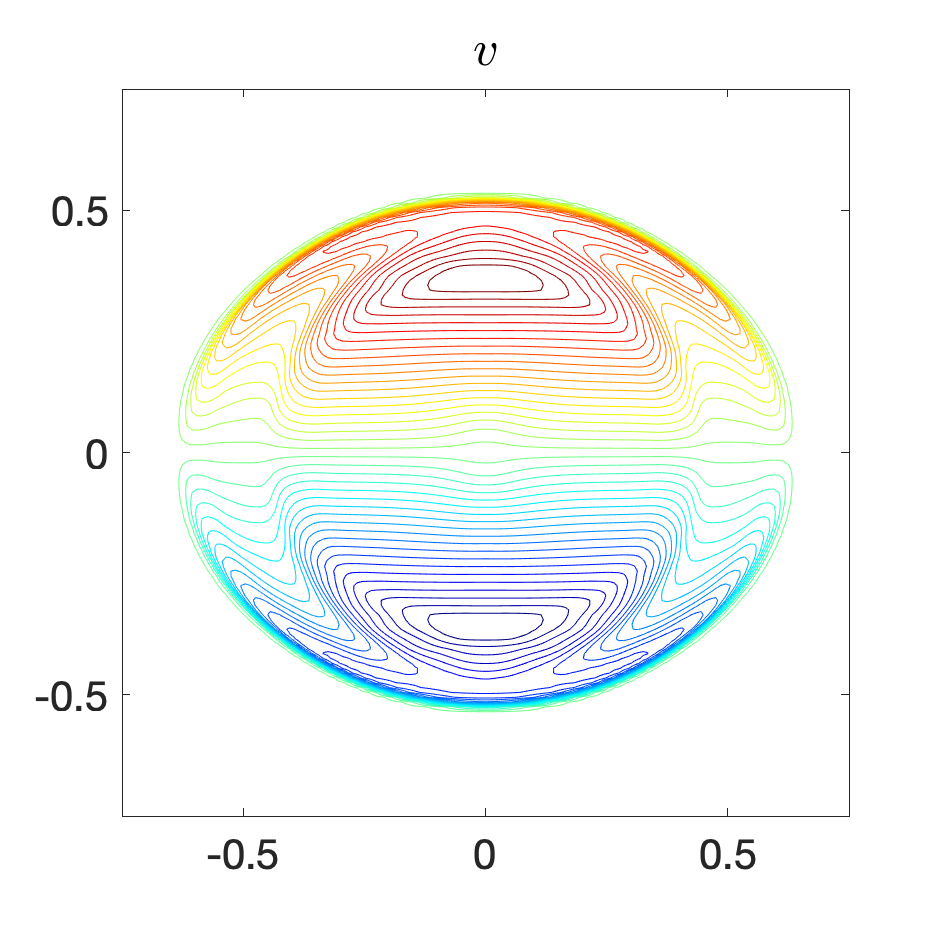}}
\vskip8pt
\centerline{\includegraphics[trim=0.6cm 1.1cm 1.4cm 0.5cm, clip, width=5.3cm]{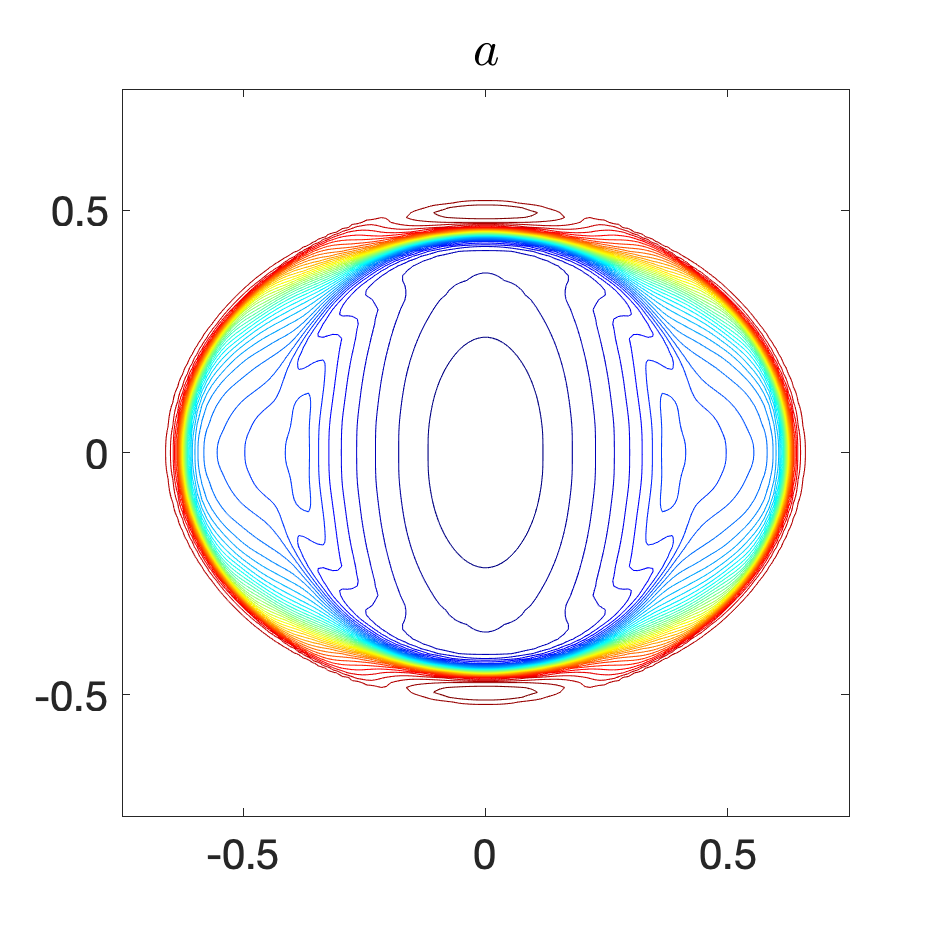}\hspace*{0.3cm}
            \includegraphics[trim=0.6cm 1.1cm 1.4cm 0.5cm, clip, width=5.3cm]{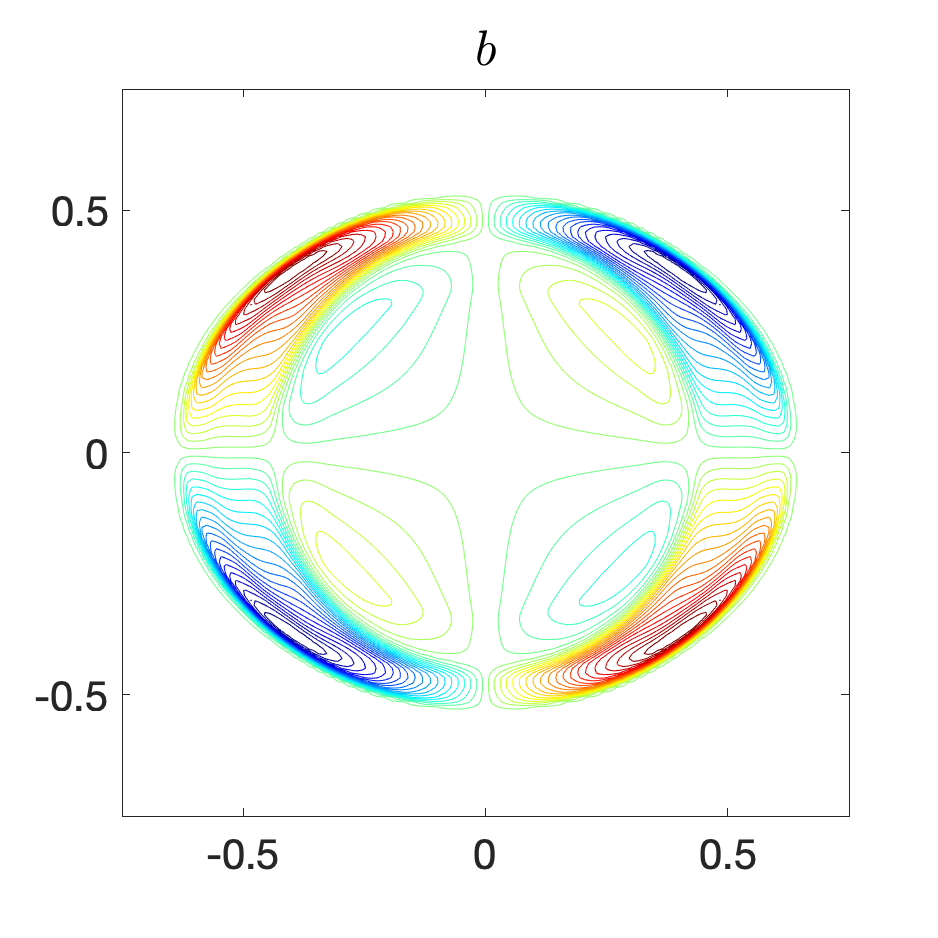}}
\caption{\sf Example 7: Fluid thickness $h$, velocities $u$ and $v$, reduced magnetic field components $a$ and $b$ computed by the proposed
PCCU scheme. 40 equally spaced contours are used in each plot.\label{fig8}}
\end{figure}

\section{Conclusion}\label{sec4}
In this paper, we have developed a new second-order unstaggered path-conservative central-upwind (PCCU) scheme for the ideal and shallow
water magnetohydrodynamics (MHD) systems. The proposed scheme is (i) locally divergence-free; (ii) Riemann-problem-solver-free; (iii)
high-resolution; (iv) robust; and (v) non-oscillatory. The derivation of the scheme is based on the Godunov-Powell nonconservative
modifications of the studied MHD systems. The local divergence-free property is enforced by augmenting the studied systems with the
evolution equations for the corresponding derivatives of the magnetic field components and by using these evolved quantities in the design
of a special piecewise linear reconstruction of the magnetic field, which also guarantees a non-oscillatory nature of the resulting scheme.
In addition, the proposed PCCU scheme allows for a proper treatment of the nonconservative product terms, which takes into account jumps of
the normal component of the magnetic field across cell interfaces, thus providing stability. The performance of the new scheme has been
illustrated on several benchmarks for both ideal and shallow water MHD systems producing high-resolution and oscillation-free results with
positive computed quantities such as density, pressure, and water depth.

In our future work, we plan to develop a provably positivity-preserving high-order PCCU scheme as well as to introduce a new well-balanced
PCCU scheme for more general shallow water MHD systems with the nonflat bottom topography and Coriolis forces are taken into account.

\section*{Acknowledgments}
The work of A. Chertock and M. Redle were partially supported by NSF grants DMS-1818684 and DMS-2208438. The work of A. Kurganov was
partially supported by NSFC grants 12111530004 and 12171226, and by the fund of the Guangdong Provincial Key Laboratory of Computational
Science and Material Design (No. 2019B030301001). The work of K. Wu was partially supported by the NSFC grant 12171227. The authors would like 
to express their gratitude to Prof. Vladimir Zeitlin from the Laboratory of Dynamical Meteorology, Sorbonne University, Ecole Normale Sup\'erieure, CNRS, Paris, for motivating and fruitful discussions. 

\bibliographystyle{siam}
\bibliography{biblio}
\end{document}